\documentclass[12pt,reqno,a4paper]{amsart}
\usepackage{
    amsmath,  amsfonts, amssymb,  amsthm,   amscd,
    gensymb,  graphicx, comment,  etoolbox, url,
    booktabs, stackrel, mathtools,enumitem, mathdots,  microtype, lmodern,    mathrsfs, graphicx, tikz,  longtable,tabularx, float, tikz, pst-node, tikz-cd, multirow, tabularx, amscd,  bm, array, makecell, diagbox, booktabs,ragged2e, caption, subcaption }
\usepackage{makecell,slashbox}
\usepackage{xcolor}
\usepackage[all]{xy}
\usepackage{graphicx}

\usepackage[utf8]{inputenc}
\usepackage{microtype, fullpage, wrapfig,textcomp,mathrsfs,csquotes,fbb}
\usepackage[colorlinks=true, linkcolor=blue, citecolor=blue, urlcolor=blue, breaklinks=true]{hyperref}
\usepackage[capitalise]{cleveref}
\usepackage{todonotes}
\usetikzlibrary{positioning}
\usetikzlibrary{shapes,arrows.meta,calc}
\usetikzlibrary{arrows}

\newtheorem{theorem}{Theorem}[section]

\newtheorem{corollary}[theorem] {Corollary}
\newtheorem{definition}[theorem]{Definition}
\newtheorem{example}[theorem]{Example}

\newtheorem{proposition}[theorem]{Proposition}
\newtheorem{remark}[theorem]{Remark}

\newcommand{\TC}{\mathrm{TC}}

\newcommand{\ct}{\mathrm{cat}}

\newcommand{\sct}{\mathrm{secat}}
\newcommand{\Top}{\mathrm{Top}}
\newcolumntype{x}[1]{>{\centering\arraybackslash}p{#1}}

\begin{document}
\title[]{On sequential versions of various parametrized invariants}
\author[N. Daundkar]{Navnath Daundkar}
\address{Department of Mathematics, Indian Institute of Technology Madras, Chennai, India.}
\email{navnath@iitm.ac.in}
% \author[J. M. Calcinas]{J. M. Garc\'ia-Calcines}
% \address{}
% \email{jmgarcal@ull.edu.es}
\author[A. Sarkar ]{Abhishek Sarkar}
\address{Department of Mathematics, Indian Institute of Science Education and Research Pune, India}
\email{abhisheksarkar49@gmail.com}
\author[A. Sarkar]{Ankur Sarkar}
\address{The Institute of Mathematical Sciences, A CI of Homi Bhabha National Institute,  Chennai, India} 
\email{ankurimsc@gmail.com}

\thanks{}

\begin{abstract} 
In this paper, we introduce and study sequential versions of several fibrewise homotopy invariants, including parametrized topological complexity, parametrized (subspace) homotopic distance. We investigate their basic properties, establish relationships among them, and compare them with the corresponding classical homotopical invariants. 
\end{abstract}

\keywords{Parametrized homotopic distance, fibrewise sectional category, fibrewise unpointed Lusternik–Schnirelman category, sequential parametrized topological complexity, fibrewise $H$-spaces}
\subjclass[2020]{55M30, 55S40, 55R70, 55U35, 55P45}
\maketitle

\section{Introduction}
In motion-planning theory, topological spaces often serve as configuration spaces that encode all allowable states of a mechanical or robotic system. Recent advances in this area have focused on invariants that capture the complexity of navigating these configuration spaces, particularly in settings where multiple stages or external parameters are involved. In this context, sequential and parametrized invariants have become central, notably the parametrized topological complexity introduced by Cohen–Farber–Weinberger \cite{C-F-W} and later extended by Farber–Paul \cite{Farber-Paul1} to the sequential setting. These invariants model algorithms that must guide a system through a prescribed sequence of states under varying external conditions, encoded by a fibration $p:E\to B$ whose fibers $X_b$ represent system states. The quantity $\TC_r[p\colon E\to B]$ measures the complexity of universal motion-planning algorithms in this parametrized, sequential context. These developments build on earlier foundational invariants, most notably the Lusternik–Schnirelmann (LS) category $\ct(X)$ \cite{L-S-cat}, Farber’s topological complexity $\TC(X)$ \cite{F}, and Rudyak’s higher (or sequential) topological complexity $\TC_r(X)$ \cite{Rudyak_sequentialTC} for a path-connected topological space $X$, which revealed deep connections between algebraic topology, critical point theory, and robot motion planning \cite{Farber_Book}. These structural parallels motivate the search for broader generalizations in various directions.

The notion of homotopic distance introduced by Macías-Virgós and Mosquera-Lois \cite{macias2022homotopic} provides a unified perspective on classical invariants such as $\ct(X)$ and $\TC(X)$, and has motivated the study of their sequential analogues, developed by Borat and Vergili \cite{Borat-Vergili}. These frameworks highlight deep structural connections among LS category, topological complexity, and their sequential versions. More recently, this line of inquiry has been developed in the parametrized (fibrewise) setting. Building on García-Calcines’s notion of fibrewise topological complexity $\TC_B(X)$ for a fibrewise space $(X,p_X)$ over $B$ \cite{Calcines-fibrewiseTC}, García-Calcines and the first author later introduced the parametrized (fibrewise) homotopic distance \cite{Calcines-Navnath}. 
They obtained cohomological and connectivity bounds, studied functorial properties such as compositions, products, and a triangle inequality, and gave estimates for fibre-preserving maps using induced fibre maps and the fibrewise LS category. A pointed version was also developed with analogous bounds and equivalence criteria.

The classical homotopy theory of fibrations was extended to the fibrewise setting by James \cite{James}, with Morries \cite{Morries-James} and Crabb \cite{C-J}, and other authors.
This article introduces the \textit{sequential version} of parametrized homotopic distance \cite{Calcines-Navnath}, extending the sequential homotopic distance of \cite{Borat-Vergili} to the fibrewise setting and providing a unified framework for various fibrewise homotopy invariants. In particular, we extends the notion of sequential parametrized topological complexity \cite{Farber-Paul1} to the fibrewise context, yielding the sequential analogue of parametrized topological complexity studied in \cite{Calcines-fibrewiseTC}. Later, we generalize these constructions to their relative counterparts. In the relative setting, Mac\'ias-Virgós \emph{et al.} introduced the \emph{subspace (relative) homotopic distance} in \cite{Relative_homotopic_distance}, generalizing both the subspace LS category and the relative topological complexity of Farber \cite{Farber_Book}. The first author introduced the notion of sequential subspace topological complexity in \cite{ND_Group_action} and subsequently Mescher–Stegemeyer formulated a parametrized version of sequential subspace topological complexity in \cite{Mescher-Stegemeyer}. In this work, we introduce the \emph{relative sequential parametrized homotopic distance}, which provides a common framework unifying all these notions.

Here, we introduce the $r$-th sequential parametrized (or fibrewise) homotopic distance of fibrewise maps $f_i \colon X \to Y$ between fibrewise spaces $X$ and $Y$ over $B$ for $i \in \{1, \dots, r\}$, denoted by $D_B(f_1,\dots,f_r)$. This is defined as the smallest integer $n \ge 0$ (or $\infty$ if no such $n$ exists) such that the fibrewise space $X$ admits a cover by $n{+}1$ open sets on which $f_1,\dots,f_r$ are fibrewise homotopic. This notion unifies and generalizes earlier concepts, including the fibrewise unpointed LS category and sequential parametrized topological complexity. In particular, if $X$ is a fibrewise space over $B$ with fibrewise projections $pr_i \colon X^r_B \to X$, then $\TC_{B,r}(X) = D_B(pr_1,\dots,pr_r)$, where $X^r_B$ is the $r$-fold fibred product of $X$ over $B$; and if $X$ is fibrewise pointed with section $s_X$, then $\ct_B^*(X) = D_B(id_X, s_X \circ p_X, \dots, s_X \circ p_X)$. When $B$ is a point, we recover the usual sequential homotopic distance. By relating parametrized homotopic distance to the fibrewise sectional category of \cite{GC}, we establish cohomological lower bounds and connectivity upper bounds, and study its behaviour under compositions, products, and fibre-preserving maps.

The paper is organized as follows.
In Section \ref{sec:prelim-fib-homotopy-theory}, we recall basic notions from fibrewise homotopy theory, including fibrewise (pointed) spaces and fibrewise fibrations and various fibrewise invariants. These concepts serve as essential ingredients for the remainder of the article.
In Section \ref{fibrewise seq parametrized TC}, we introduce the sequential analogue of parametrized topological complexity in the fibrewise setting and establish its fibrewise homotopy invariance (Proposition \ref{fibrewise_homotopy_invariant}). We also extend the classical LS category–topological complexity inequality to the fibrewise pointed setting (Proposition \ref{relation_SFTC_Cat}).  \ref{fibrewise pointed seq TC} compares the pointed and unpointed versions of the sequential fibrewise parametrized topological complexity. In Section \ref{sec:sphd}, which forms the central theme of the article, introduces the sequential analogue of parametrized homotopic distance. This notion unifies several numerical topological invariants within a broader framework. We provide cohomological lower bounds and homotopy-dimension upper bounds for maps over fibrant spaces. Section \ref{properties of seq PHD} establishes key properties of this invariant, including its relation to the fibrewise LS category of maps (Proposition \ref{seq PHD-LS cat}). \ref{composition of maps} examines its behavior under composition of fibrewise maps. We also prove a triangle inequality on normal spaces (Theorem~\ref{sub-additive}) and compare its values on products with those on the individual fibrewise maps (Proposition \ref{D_B_on_product}).
In Section \ref{seq PHD fibrewise fibration}, we study this invariant for fibrewise fibrations. Section \ref{seq Pointed PHD} develops the pointed version for fibrewise pointed maps, and \ref{seq PHD_comparision} compares the pointed and unpointed theories. Finally, Section \ref{Relative seq PHD} introduces a relative (subspace) version of sequential homotopic distance in the fibrewise setting and explore some its properties.
\section{Background on parametrized invariants}\label{sec:prelim-fib-homotopy-theory}
In this section, we recall the notation and standard results from fibrewise homotopy theory and related invariants that will be used throughout the paper. For further details, see \cite{James, C-J, GC}.

\subsection{Fibrewise homotopy theory}
A fibrewise space over \( B \) consists of a pair \( (X, p_X) \), where \( X \) is a topological space and \( p_X \colon X \to B \) is a projection map. When the context is clear, we simply write \( X \) for such a fibrewise space. A fibrewise map \( f \colon X \to Y \) between fibrewise spaces is a continuous map satisfying \( p_Y \circ f = p_X \). 

For fibrewise spaces $X$ and $Y$, their fibrewise product is defined as 
\[
X \times_B Y = \{(x,y)\in X\times Y \mid p_X(x)=p_Y(y)\}.
\]

Let $I=[0,1]$ be the unit interval. The fibrewise cylinder of a fibrewise space $X$ over $B$ is the product $X\times I$ with the projection $X\times I \xrightarrow{pr_1} X \xrightarrow{p_X} B$ and is denoted by $I_B(X)$. This yields the usual notions of fibrewise homotopy $\simeq_B$ between fibrewise maps and fibrewise homotopy equivalence.

For a fibrewise space \(X\) over \(B\), the fibrewise free path space or fibrewise cocylinder \(P_B(X)\) of $X$ is defined by the pullback diagram in the category of topological spaces
\begin{equation}%\label{eq:PBX}
 \xymatrix{
 P_B(X) \ar[r] \ar[d] & X^I \ar[d]^{p_X^I} \\
 B \ar[r]_c & {B^I,}
 }
\end{equation}
where \(X^I\) and \(B^I\) denote the free path spaces with the compact–open topology, \(p_X^I\) is induced by \(p_X\), and \(c\colon B\to B^I\) assigns to each point its constant path. Thus,
\[
P_B(X)=\{(b,\alpha)\in B\times X^I \mid p_X^I\circ \alpha = c_b\},
\]
equipped with the projection \((b,\alpha)\mapsto b\), where $c_b$ denotes the constant path in $B^I$.

We now recall the notion of a fibrewise fibration. A fibrewise map $p\colon E \to Y$ is called a fibrewise fibration if it satisfies the homotopy lifting property with respect to any fibrewise space. That is, for any commutative diagram in in the category of fibrewise spaces, there exists a lift $I_B(Z) \to E$  making the following diagram commute:  
\[\begin{tikzcd}
	Z && E \\
	{I_B(Z)} && Y.
	\arrow["f", from=1-1, to=1-3]
	\arrow["{i_0}"', hook, from=1-1, to=2-1]
	\arrow["p", from=1-3, to=2-3]
	\arrow[dashed, from=2-1, to=1-3]
	\arrow["H"', from=2-1, to=2-3]
\end{tikzcd}\]

The fibrewise map
\begin{align} \label{fibrewise path space map}
  \Pi_X=(d_0,d_1)\colon P_B(X)\to X\times_B X,  
\end{align}
 defined by \((b,\alpha) \mapsto (\alpha(0),\alpha(1))\),
is always a fibrewise fibration, though not necessarily a Hurewicz fibration. However, when \(p_X\colon X \to B\) is a Hurewicz fibration, one can verify that \(\Pi_X\) is also a Hurewicz fibration.

We now recall the pointed counterparts of the unpointed notions reviewed above. A fibrewise pointed space over $B$ is a triple $(s_X, X, p_X)$, where $(X, p_X)$ is a fibrewise space and $s_X \colon  B \to X$ is a section of $p_X$; we will simply write $X$ when this causes no ambiguity. A fibrewise pointed map $f \colon  X \to Y$ between fibrewise pointed spaces $X$ and $Y$ is a fibrewise map satisfying $f \circ s_X = s_Y$.

A subspace \(A \subseteq X\) containing the section \(s_X(B)\) is a fibrewise pointed space, called a fibrewise pointed subset of \(X\), with the inclusion \(A \hookrightarrow X\) being fibrewise pointed.

For a fibrewise pointed space $X$ over $B$, the fibrewise pointed cylinder is defined as the pushout
\[
\xymatrix{
B \times I \ar[r]^{pr} \ar[d]_{s_X \times id} & B \ar[d] \\
X \times I \ar[r] & I_B^B(X),
}
\]
and as in the unpointed setting, this gives the notion of fibrewise pointed homotopy between fibrewise pointed maps, written $\simeq_B^B$. A fibrewise pointed homotopy $F \colon  I_B^B(X) \to Y$ may equivalently be described as a fibrewise homotopy $F' \colon  I_B(X) \to Y$ satisfying $F'(s_X(b),t) = s_X(b)$ for all $b \in B$ and $t \in I$. The corresponding notion of fibrewise pointed homotopy equivalence is defined in the usual way.

The pointed fibrewise cocylinder of $X$ over $B$ 
 is the fibrewise space $P_B(X)$ together with the section $(id_B, c\circ s_X)\colon  B \to P_B(X)$, denoted by $P_B^B(X),$ which induced by the pullback property

\subsection{Fibrewise sectional category and induced invariants} 
In this subsection, we recall various numerical invariants associated to fibrewise spaces.

We begin by recalling the fibrewise unpointed version of the LS category introduced in \cite{iwase2010topological}.
\begin{definition}\label{fibrewise_unpointed_LS_Cat}
    Let $X$ be a fibrewise pointed space over $B$. The fibrewise unpointed LS category of $X$, denoted $\ct_B^*(X)$, is the least nonnegative integer $n$ such that there exists an open cover $\{U_0, U_1, \dots, U_n\}$ of $X$ with the property that, for each $0 \leq i \leq n$, the following diagram commutes up to fibrewise homotopy:
    \[
    \begin{tikzcd}
    	{U_i} && X \\
    	& {B.}
    	\arrow["{i_{U_i}}", hook, from=1-1, to=1-3]
    	\arrow["{p_X\vert_{U_i}}"', from=1-1, to=2-2]
    	\arrow["{s_X}"', from=2-2, to=1-3]
    \end{tikzcd}
    \]
    If no such integer exists, we set $\ct_B^*(X) = \infty$.
\end{definition}
Iwase and Sakai \cite{iwase2010topological} proved that for fibrewise well-pointed spaces the unpointed and pointed fibrewise LS categories coincide. Moreover, if $B$ is a locally finite simplicial complex, then consider the fibrewise pointed space $d(B)$ with $p_{d(B)} \colon  B \times B \to B$ given by the second projection and the diagonal map as a section $s_{d(B)}$. In this context, they proved that Farber's topological complexity $\mathrm{TC}(B)$ is recovered as the fibrewise LS category of $d(B)$.

As another fibrewise invariant extending the unpointed LS category, García--Calcines introduced a fibrewise analogue of sectional category \cite[Definition~2.1]{GC}. This invariant admits equivalent Whitehead and Ganea type descriptions in the fibrewise setting and measures the minimal complexity required to obtain fibrewise homotopy sections of a given map over the base.
\begin{definition}
The fibrewise sectional category of a fibrewise map $f\colon  E \to X$ over $B$, denoted by $\operatorname{secat}_B(f)$, is the smallest non-negative integer $k$ for which there exists an open cover $\{U_0, \dots, U_k\}$ of $X$ such that for each $U_i$ there is a fibrewise homotopy section $s_i\colon  U_i \to E$, that is,
\[
f \circ s_i \simeq_B i_{U_i},
\]
where $i_{U_i}\colon  U_i \hookrightarrow X$ is the inclusion map. If  no such $k$ exists, then $\operatorname{secat}_B(f)=\infty.$
\end{definition}

The fibrewise pointed cohomology was introduced by  Garcia-Calcines and the first author in \cite{Calcines-Navnath}. We now recall and state the cohomological lower bound on the fibrewise pointed sectional category of a fibrewise map.

Note that there is a split short exact sequence of graded abelian groups
$$
\xymatrix{
{0} \ar[r] & {H^*(B)} \ar[r]^{p_X^*} & {H^*(X)} \ar@/^1pc/[l]^{s_X^*} \ar[r] & {H^*(X) / \langle p_X^*(H^*(B)) \rangle} \ar[r] & {0}
}
$$
inducing an isomorphism
$$
H^*(X)/\langle p_X^*(H^*(B))\rangle \cong \mathrm{ker}(s_X^*).
$$
Moreover, if $H^*(X,s_X(B))$ denotes the cohomology of the pair $(X,s_X(B))$, then the following isomorphism holds:
$$
H^*(X,s_X(B))\cong \mathrm{ker}(s_X^*).
$$

Indeed, let $i\colon s_X(B)\to X$ and $j\colon (X,\emptyset)\to (X,s_X(B))$ denote the inclusions. Consider the following segment of the long exact sequence in cohomology associated with the pair $(X,s_X(B))$:
% {\color{red}$$
% \xymatrix{
% {H^{*+1}(X)} \ar[r]^{i^*} & {H^{*+1}(s_X(B))} \ar[r]^{\delta ^*} & {H^*(X, s_X(B))} \ar[r]^(.6){j^*} & {H^*(X)} \ar[r]^(.4){i^*} & {H^*(s_X(B)).}
% }
% $$}
\[
\begin{tikzcd}
{H^{*}(X,s_X(B))} \arrow[r, "j^*"] & {H^{*}(X)} \arrow[r, "i^*"] & {H^*(s_X(B))} \arrow[r, "\delta^*"] & {H^{*+1}(X,s_X(B))} \arrow[r, "j^*"] & {H^{*+1}(X).}
\end{tikzcd}
\]

Since $i^*$ is surjective in all dimensions, it follows that $\delta^*=0$, and thus $j^*$ is injective. Consequently,
$$
\mathrm{ker}(s_X^*)=\mathrm{ker}(i^*)=\mathrm{im}(j^*)\cong H^*(X,s_X(B)).
$$

The above considerations lead to the following formulation of fibrewise pointed cohomology.

\begin{definition}\label{def:fib-cohomology}
Let $X$ be a fibrewise pointed space over $B$. The fibrewise
pointed cohomology of $X$ (with coefficients in $R$) is defined as
$$
H_B^*(X):=H^*(X,s_X(B)).
$$
\end{definition}

We now state the cohomological lower bound on the fibrewise pointed sectional category. Consider the induced ring homomorphism
$$
\pi^*\colon H_B^*(X)\to H_B^*(E),
$$
arising from the map of pairs $\pi\colon (E,s_E(B))\to (X,s_X(B))$.

\begin{theorem}\label{thm:coho-lb-secatB}
Let $\pi\colon E\to X$ be a fibrewise pointed map. Then
$$
nil(\mathrm{ker}(\pi^*))\leq \sct_B^B(\pi),
$$
where $\sct_B^B(\pi)$ denotes the fibrewise pointed sectional category of the map $\pi$.
\end{theorem}

We next turn to another fibrewise invariant that plays a central role in parametrized motion planning. The parametrized topological complexity \cite{C-F-W} was later extended  to arbitrary fibrewise spaces \cite[Definition~8]{Calcines-fibrewiseTC} using the formulation adopted here. We now state the precise definition of fibrewise topological complexity for a fibrewise space over $B$.

\begin{definition}
The fibrewise topological complexity of a fibrewise space \((X,p_X)\) over \(B\), denoted \(\TC_B(X)\), is the fibrewise sectional category of the fibrewise fibration \(\Pi_X \colon  P_B(X) \to X \times_B X\) introduced in (\ref{fibrewise path space map}). Thus,
\[
\TC_B(X) := \sct_B(\Pi_X).
\]
\end{definition}

\section{Formal aspects of sequential parametrized topological complexity} \label{fibrewise seq parametrized TC}

García–Calcines introduced the notion of parametrized topological complexity for fibrewise spaces in \cite{Calcines-fibrewiseTC}. In this section, we develop its sequential analogue. This new invariant incorporates both the original fibrewise parametrized topological complexity and the sequential parametrized topological complexity. We also establish several basic formal properties in this setting and conclude by presenting the pointed variant.

\subsection{Sequential fibrewise topological complexity}

To introduce the sequential parametrized topological complexity, we begin by recalling the standard fibrewise fibration obtained by evaluating fibrewise paths at prescribed times. \\
For a fibrewise space $E$ over $B$, there is a fibrewise fibration
\[
\Pi_{r,E} \colon P_B(E) \longrightarrow E_B^{\,r},
\]
defined on each fibre by sending a path $\gamma$ to the $r$-tuple of its evaluations:
\begin{equation}\label{evaluation}
\Pi_{r,E}(b,\gamma)
= \bigl(\gamma(0),\, \gamma(\frac{1}{r-1}),\,\dots,\,\gamma(1)\bigr),
\qquad (b,\gamma)\in P_B(E).
\end{equation}

\begin{definition}\label{SPTC}
The sequential parametrized topological complexity of a fibrewise space $E$ over $B$ is defined by
\[
\TC_{B,r}(E) := \sct_B(\Pi_{r,E}).
\]
\end{definition}

%\begin{remark}\begin{itemize} \item[(A)] 
    A fibrewise space $E$ is fibrant when the projection $p_E \colon E \to B$ is a Hurewicz fibration. Under this assumption, both $P_B(E)$ and $E_B^{\,r}$ are fibrant. By \cite[Theorem~2.10]{GC}, it follows that $\TC_{B,r}(E)=\sct(\Pi_{r,E})$.\\
    %\item[(B)] 
    Moreover, the fibrewise map $c\colon  E\to P_B(E)$ given by $c(e)= (P_E(e),c_e)$ is a fibrewise homotopy equivalence $($see \cite[Remark 10]{Calcines-fibrewiseTC}$)$. Using the fact, we obtain the following commutative diagram: 
    \[\begin{tikzcd}
	E && {E_B^r} \\
	& {P_B(E).}
	\arrow["{\Delta_E^r}", from=1-1, to=1-3]
	\arrow["c"', "\simeq_B",from=1-1, to=2-2]
	\arrow["{\Pi_{r,E}}"', from=2-2, to=1-3]
\end{tikzcd}\] 
Therefore, $\TC_{B,r}(E)=\sct_B(\Delta_E^r)$ by \cite[Proposition 5]{Calcines-fibrewiseTC}, where $\Delta_E^r$ is the diagonal map.
%\end{itemize}\end{remark}
Now, we show that $\TC_{B,r}(-)$ is a fibrewise homotopy invariant.
\begin{proposition}\label{fibrewise_homotopy_invariant}
    Let $E$ and $\widetilde{E}$ be two fibrewise homotopy equivalent spaces over $B$. Then \[\TC_{B,r}(E)=\TC_{B,r}(\widetilde{E}).\]
\end{proposition}
\begin{proof}
    Let $f\colon  E\to \widetilde{E}$ be fibrewise homotopy equivalence. This induces a fibrewise homotopy equivalence $P_B(f)\colon  P_B(E)\to P_B(\widetilde{E})$. Let  $f^r_B$ be the $r$-times fibred product of the map $f$. Thus, we obtain the induced fibrewise homotopy equivalence $f^r_B\colon  E_B^r\to \widetilde{E}_B^r$.%(see \cite{B})\[f^r_B\colon  E_B^r\to \widetilde{E}_B^r.\]
    Then the result follows by applying \cite[Proposition 5]{Calcines-fibrewiseTC} to the following commutative diagram 
    \[\begin{tikzcd}
	{P_B(E)} && {P_B(\widetilde{E})} \\
	{E_B^r} && {\widetilde{E}_B^r.}
	\arrow["{P_B(f)}", "\simeq_B"',from=1-1, to=1-3]
	\arrow["{\Pi_{r,E}}"', from=1-1, to=2-1]
	\arrow["{\Pi_{r,\widetilde{E}}}", from=1-3, to=2-3]
	\arrow["{f^r_B}"', "\simeq_B",from=2-1, to=2-3]
\end{tikzcd}\]
\end{proof}
If $E$ is a fibrewise pointed space, we obtain the following result.
\begin{corollary}\label{contractibility}
    Let $E$ be a fibrewise pointed space over $B$. Then $\TC_{B,r}(E)=0$ if and only if $E$ is fibrewise contractible. 
\end{corollary}
\begin{proof}
    Suppose $E\simeq_B B$. Then by Proposition \ref{fibrewise_homotopy_invariant}, we get $$\TC_{B,r}(E)=\TC_{B,r}(B)=\sct_B(Id_B)=0.$$

Suppose $\TC_{B,r}(E) = 0$. Then there exists a fibrewise map $\sigma \colon  E_B^r \to P_B(E)$ such that the following diagram commutes up to fibrewise homotopy:
\[
\begin{tikzcd}
E_B^r \arrow[rr, "{Id_{E_B^r}}"] \arrow[dr, "\sigma"'] & & E_B^r \\
& {P_B(E).} \arrow[ur, "{\Pi_{r,E}}"'] &
\end{tikzcd}
\]
Here, the map $\sigma \colon E_B^r \to P_B(E)$ is defined by 
\[
\sigma(x_1, \dots, x_r) = \big(p_E(x_1), \bar{\sigma}(x_1, \dots, x_r)\big),
\]
where $\bar{\sigma}(x_1, \dots, x_r)$ is a path in $E$ satisfying 
\[
\bar{\sigma}(x_1, \dots, x_r)\!\left(\frac{i}{r-1}\right) = x_{i+1}, \quad
p_E \circ \bar{\sigma}(x_1, \dots, x_r) = c_{p_E(x_1)}.
\]

Next, note that the fibrewise map $p = p_{P_B(E)} \colon P_B(E) \to B$ admits a fibrewise homotopy inverse $q \colon B \to P_B(E)$, defined by 
$q(b) = (b, c_{s_E(b)}),$
where $s_E \colon B \to E$ is the section of $p_E$. Indeed, $p \circ q = Id_B$, and $q \circ p \simeq_B Id_{P_B(E)}$ through the fibrewise homotopy $$H \colon I_B(P_B(E)) \rightarrow P_B(E),$$
defined by
$H (b, \alpha, t) = (b, \overline{H}(b, \alpha, t)),$
where $$\overline{H}(b, \alpha, t)(s) := \bar{\sigma}\big(\alpha(st), \underbrace{s_E(b), \dots, s_E(b)}_{(r-1)\text{-times}}\big)(1-t).$$
Hence $P_B(E) \simeq_B B$. Since $E \simeq_B P_B(E)$ by \cite[Remark 10]{Calcines-fibrewiseTC}, the result follows.
\end{proof}

%%%%%%%%%

For a fibrewise pointed space \(E\) over \(B\), one has  
\[
\ct_B^*(E) \leq \TC_{B}(E) \leq \ct_B^*(E \times_B E)
\]
(see \cite[Proposition 13]{Calcines-fibrewiseTC}). We now establish the sequential analogue of this inequality.

\begin{proposition}\label{relation_SFTC_Cat}
    Let $E$ be a fibrewise pointed space over $B$. Then 
    \[\ct_B^*(E_B^{r-1})\leq \TC_{B,r}(E)\leq \ct_B^*(E_B^r).\]
\end{proposition}
\begin{proof}
    The inequality $\TC_{B,r}(E)\leq \ct_B^*(E_B^r)$ follows directly from the Definition~\ref{SPTC} and \cite[Proposition 2.3]{GC}.

    To prove the other inequality, we consider $U\subseteq E_B^r$ an open subset with a fibrewise homotopy section $s\colon U\to P_B(E)$ of $\Pi_{r,E}\colon P_B(E)\to E_B^r$. Then, $s$ has the expression \[s(x_1,\dots,x_r)=\left(p_E(x_1),\bar{s}(x_1,\dots,x_r)\right),\]
    where the map $\bar{s} \colon  U \to E^I$ satisfies $\bar{s}(x_1, \dots, x_r)\!\left(\tfrac{i}{r-1}\right) = x_{i+1} \text{ for } i = 0, \dots, r-1,$ and $p_E \circ \bar{s}(x_1, \dots, x_r) = c_{p_E(x_1)}\text{ for all } (x_1, \dots, x_r) \in U$. Now, we consider an open subset \[V:=\{(x_1,\dots,x_{r-1})\mid \left(x_1,(s_E\circ p_E)(x_1),\dots,(s_E\circ p_E)(x_{r-1})  \right)\in U\}\] of $E_B^{r-1}$ and define the fibrewise homotopy $H\colon I_B(V)\to E_B^{r-1}$ by 
    \[H(x_1,\dots,x_{r-1},t)= \left(\kappa(t), \dots,\kappa(\frac{r-1-i}{r-1}t+\frac{i}{r-1}), \dots,\kappa(\frac{1}{r-1}t+ \frac{r-2}{r-1})\right),\] where $\kappa(t)=\bar{s}(x_1,\dots,x_{r-1},(s_E\circ p_E)(x_1))(t)$.
    Thus, $H(x_1,\dots,x_{r-1}, 0)= (x_1, \dots, x_{r-1})$ and $H(x_1,\dots,x_{r-1}, 1)= ((s_E\circ p_E)(x_1), \dots, (s_E\circ p_E)(x_1))$.
This implies that $V$ is fibrewise categorical. Now the result follows using this argument with open cover.
\end{proof}

We now consider a continuous map $\lambda\colon B'\to B $ and a fibrewise space $E$ over $B$. Then \(\lambda\) determines a fibrewise space \(E'\) over \(B'\), fitting into the pullback diagram below:
\[\begin{tikzcd}
	{E'} & E \\
	{B'} & {B.}
	\arrow[from=1-1, to=1-2]
	\arrow["{p_{E'}}"', from=1-1, to=2-1]
	\arrow["{p_E}", from=1-2, to=2-2]
	\arrow["\lambda"', from=2-1, to=2-2]
\end{tikzcd}\]
Thus, 
\[
E' = \{(b', e) \in B' \times E \mid \lambda(b') = p_E(e)\}~~\text{and}~~ p_{E'}(b', e) = b'.
\]
%with projection \(p_{E'}(b', e) = b'\).  

Moreover, any fibrewise map \(f\colon  E \to X\) over \(B\) induces a fibrewise map \(f'\colon E' \to X'\) over \(B'\),  defined by \(f'(b', e) = (b', f(e))\), where $X'=\{(b', x)\in B'\times X\mid \lambda(b')= p_X(x) \}$. %Thus, \(\lambda\) gives rise to a functor $\lambda^* \colon \Top_B \to \Top_{B'}$. 
Now, using \cite[Proposition 14]{Calcines-fibrewiseTC}, we obtain the following proposition.
\begin{proposition}
    Let $\lambda\colon B'\to B$ be a map and $E$ be a fibrewise space over $B$. Then \[\TC_{B,r}(\lambda^*(E))\leq \TC_{B,r}(E),\] where $\lambda^*(E)$ is the fibrewise pullback of $E$.
\end{proposition}
\begin{corollary}
    Let $E$ be a fibrewise space over $B$. If $B'\subseteq B$ and $E'=(p_E)^{-1}(B')$, then \[\TC_{B',r}(E')\leq \TC_{B,r}(E).\]
\end{corollary}
\begin{proof}
We note that the space \(E'\) is the pullback of \(p_E\) along the inclusion \(B' \hookrightarrow B\). The statement then follows from the preceding proposition.
\end{proof}

We establish a product inequality for sequential parametrized topological complexity. We first recall the following consequence of \cite[Proposition 3.11]{GCrelsecat}.
\begin{proposition}\label{Bound of section category of product}
    Let $f: E\to X$ and $f':E'\to X'$ be fibrewise maps over $B$. Moreover assume that both $X$ and $X'$ are normal spaces. Then 
    \[\text{ max }\{\sct_B(f),~\sct_{B}(f')\} \leq \sct_{B}(f\times f' )\leq \sct_B(f)+\sct_{B}(f') .\]
\end{proposition}

We now state the following result.
\begin{proposition}\label{parametrized_TC_on_product}
    Let $E$ and $\widetilde{E}$ be fibrewise spaces over $B$ such that $E, \widetilde{E}$ and $B$ are metrizable. Then 
    \[\text{ max }\{\TC_{B,r}(E), \TC_{B,r}(\widetilde{E})\}\leq \TC_{B,r}(E\times_B \widetilde{E})\leq \TC_{B,r}(E)+ \TC_{B,r}(\widetilde{E}).\]
\end{proposition}
\begin{proof}
    Let $E''=E \times_{B}E'$.
Note that, we obtain the product equalities
$$P_{B}(E'')= P_B(E)\times_{B} P_{B}({E}'),~\text{and}~ (E\times_{B} {E}')_{B}^r= E^r_B\times_{B} {E'}_{B}^r$$
and the fibration $\Pi_{r,E''}\colon P_{B}(E'')\to {E''}_{B}^r$ is equivalent to the product fibration \[\Pi_{r,E}\times_{B} \Pi_{r,{E}'}\colon P_B(E)\times_{B} P_{B}({E}')\to E_B^r\times_{B} {E'}_{B}^r.\]
 The conclusion then follows from Proposition~\ref{Bound of section category of product}.
\end{proof}

\subsection{Fibrewise pointed version of sequential topological complexity} \label{fibrewise pointed seq TC}

To define the sequential fibrewise pointed topological complexity, we begin by recalling the notion of the fibrewise pointed sectional category, originally introduced by Garc\'ia-Calcines \cite{GC}.

\begin{definition} $(${\cite[Definition~3.5]{GC}}$)$
Let \(f\colon E \to X\) be a fibrewise map over a base space \(B\). The fibrewise pointed sectional category of \(f\), denoted by \(\sct_B^B(f)\), is the least integer \(k \ge 0\) such that \(X\) admits an open cover \(\{U_0, \dots, U_k\}\) with the property that for each \(i\), there exists a fibrewise pointed map \(s_i\colon U_i \to E\) satisfying \(f \circ s_i \simeq_B^B i_{U_i}\).
\end{definition}
For a fibrewise pointed space \(E\) over \(B\), consider the fibrewise pointed fibration (similar to \eqref{evaluation})
$$\Pi_{r,E}^B\colon P_B^B(E) \to E_B^r.$$  Using the notion of fibrewise pointed sectional category, we define the following.

\begin{definition}
Let \(E\) be a fibrewise pointed space over \(B\). The sequential fibrewise pointed topological complexity of \(E\) is
\[
\TC_{B,r}^{B}(E) := \sct_B^B(\Pi_{r,E}^B).
\]
\end{definition}

Using arguments similar to those employed for $\TC_{B,r}(-)$, one can verify that $\TC_{B,r}^B(-)$ satisfies analogous properties.
\begin{proposition}
    Let $E$ be a fibrewise pointed space over $B$.
    \begin{itemize}
        \item[(i)] $\TC_{B,r}^B(E)$ can be equivalently defined as 
        $\TC_{B,r}^B(E):=\sct_B^B(\Delta_E^r).$
        \item[(ii)] If $E\simeq_B^B \widetilde{E}$ are fibrewise pointed homotopy equivalent, then $\TC_{B,r}^B(E)=\TC_{B,r}^B(\widetilde{E}).$ 
        \item[(iii)] $\TC_{B,r}^B(E)=0$ if and only if $E\simeq_B^B B$.
        \item[(iv)]  %Denote the fibrewise pointed LS category of $X$ by $\ct_B^B(X)$ $($see \cite{GC}$)$. We obtain the inequalities
        $\ct_B^B(E_B^{r-1})\leq \TC_{B,r}^B(E)\leq \ct_B^B(E_B^r)$, where $\ct_B^B(X)$ denotes the fibrewise pointed LS category of $X$ in the sense of \cite{Morries-James}.
        \item[(v)] Let $\lambda\colon  B'\to B$ be a continuous map and $E$ be a fibrewise space over $B$. Then the pullback construction gives rise to a functor $\lambda^*\colon  \Top(B')\to \Top(B)$ such that $\TC_{B',r}^{B'}(\lambda^*(E))\leq \TC_{B,r}^B(E).$
    \end{itemize}
\end{proposition}

As a consequence of \cite[Theorem 4.1]{GC}, we obtain the following inequality.
\begin{corollary}\label{comparision_various_tc}
    Let $E$ be a fibrewise pointed space over $B$. If $E$ is fibrewise locally equiconnected and $E_B^r$ is a normal space, then $\TC_{B,r}(E)\leq \TC_{B,r}^B(E)\leq \TC_{B,r}(E)+1.$
\end{corollary}
%%%%%%%%%%%%%%%%%%%%%%%%%%%%%%%%%%%%%%%%%%%%%%%%%%%%%%%%%%%%%%%%%%%%%%%%%%%%

\section{Sequential parametrized homotopic distance}\label{sec:sphd}

We here introduce the central concept of our study, the sequential version of the notion parametrized homotopic distance introduced in \cite{Calcines-Navnath}. When the base of fibrewise maps is a point, we can recover sequential homotopic distance of maps, introduced in \cite{Borat-Vergili}. Along with the key initial properties, comparison with the fibrewise versions of the Lusternik-Schnirelman (LS) category and sequential homotopy distance of maps will be explored in this section.

\subsection{Sequential parametrized homotopic distance} 
We introduce the notion of sequential parametrized homotopic distance and interpret it through the fibrewise sectional category. In particular, we obtain a cohomological lower bound for the sequential parametrized homotopic distance of fibrewise maps between fibrant spaces.
\begin{definition}\label{seqential_parametrized_homotopic distance}
Let $f_i \colon X\to Y$ be fibrewise maps between fibrewise spaces $X$ and $Y$ over $B$, for $i \in \{1, \dots,r\}$. The sequential parametrized homotopic distance of $f_1,\dots,f_r$, is denoted by $D_B(f_1,\dots, f_r)$ and is defined as the smallest non-negative integer $n$ for which there exists an open cover $\{U_0,\dots, U_n\}$ of $X$ such that $f_s|_{U_i}\simeq_{B}f_t|_{U_i}$ for $0\leq i\leq n$ and $1\leq s,t\leq r$.  If no such open cover exists, we set $D_B(f_1,\dots,f_r)=\infty$. 
\end{definition}
The sequential parametrized homotopic distance thus serves as a measure of how close finitely many fibrewise maps are to being fibrewise homotopic.

 Given fibrewise maps $f_1,f_2,\dots,f_r\colon X\to Y$ over $B$, we consider the following pullback diagram
\begin{equation}\label{Pullback_Diagram}
\begin{tikzcd}
	{\mathcal{P}_B(f_1,f_2,\dots,f_r)} && {P_B(Y)} \\
	X && {Y_B^r}.
	\arrow[from=1-1, to=1-3]
	\arrow["{\widetilde{\Pi}_{r,Y}}"', from=1-1, to=2-1]
	\arrow["{\Pi_{r,Y}}", from=1-3, to=2-3]
	\arrow["{(f_1,\dots,f_r)}"', from=2-1, to=2-3]
\end{tikzcd}\end{equation}
Here the map $\Pi_{r,Y}\colon P_B(Y)\to Y_B^r$ is given by $\Pi_{r,Y}(b,\gamma)=\left(\gamma(0),\gamma(\frac{1}{r-1}),\dots,\gamma(1)\right)$ and $(f_1,\dots,f_r)(x)= (f_1(x),\dots,f_r(x))$. Thus, $\mathcal{P}_B(f_1,\dots,f_r)$ has the following expression 
\[\mathcal{P}_B(f_1,\dots,f_r)=\{\left(x, (b,\gamma)\right)\in X\times P_B(Y)\mid f_{i+1}(x)=\gamma(\frac{i}{r-1}),~~i=0, \cdots, r-1\}.\]
Since $\Pi_{r,Y}$ is a fibrewise fibration, $\widetilde{\Pi}_{r,Y}$ is also, as they are connected by a pullback.

We now proceed to establish a connection between the sequential parametrized homotopic distance and the fibrewise sectional category.
\begin{proposition}\label{homotopic_distance_and_secat}
    Let $f_1,f_2,\dots,f_r\colon X\to Y$ be fibrewise maps over $B$. Then 
    \[D_B(f_1,\dots,f_r)= secat_B(\widetilde{\Pi}_{r,Y}).\]
\end{proposition}
\begin{proof}
Let $U \subseteq X$ be an open subset such that 
\[
f_1|_{U} \simeq_B f_2|_{U} \simeq_B \cdots \simeq_B f_r|_{U}.
\]
Thus, there exist fibrewise homotopies
$H_1\colon f_1|_{U} \simeq_B f_2|_{U}, \ \dots, \ H_{r-1}\colon f_{r-1}|_{U} \simeq_B f_r|_{U}.$
By concatenating these, we obtain a fibrewise homotopy 
$H\colon I_B(U) \longrightarrow Y$
such that 
$$
H\!\left(u, \tfrac{i}{r-1}\right) = f_{i+1}(u),  \quad i = 0, \dots, r-1,
$$ for all ${u \in U}$.
Let $s\colon U \to P_B(Y)$ denote the adjoint of $H$. By the universal property of the pullback, there exists a unique map 
$\sigma\colon U \to \mathcal{P}_B(f_1, \dots, f_r) $
such that the following diagram commutes upto fibrewise homotopy:
\[
\begin{tikzcd}
	U \arrow[drr, bend left, "s"] \arrow[ddr, bend right, hook', "i_U"'] 
	\arrow[dr, dotted, "\sigma"] & & \\
	& {\mathcal{P}_B(f_1,\dots,f_r)} \arrow[r] \arrow[d, "\widetilde{\Pi}_{r,Y}"'] 
	& P_B(Y) \arrow[d, "\Pi_{r,Y}"] \\
	& X \arrow[r, "{(f_1,\dots,f_r)}"'] & {Y_B^r.}
\end{tikzcd}
\]
From the diagram above we see that $\widetilde{\Pi}_{r,Y}\circ\sigma \simeq_B i_U$. Therefore, by applying this argument for open covers, we obtain $ \sct_B(\widetilde{\Pi}_{r,Y})\leq D_B(f_1,\dots,f_r).$

Let $\sigma\colon U\to \mathcal{P}_B(f_1,\dots,f_r)$ be a fibrewise homotopy section of $\widetilde{\Pi}_{r,Y}$. Next, consider the composition of $\sigma$ with the fibrewise map $\mathcal{P}_B(f_1,\dots,f_r)\to P_B(Y)$. This composition produces fibrewise homotopy between $f_s\vert_U$ and $f_t\vert_U$ for all $1\leq s,t\leq r$. Again, using the same argument to an open cover of $X$, we obtain $D_B(f_1,\dots,f_r)\leq \sct_{B}(\widetilde{\Pi}_{r,Y})$. This completes the proof.
\end{proof}
Since $\widetilde{\Pi}_{r,Y}$ is the fibrewise pullback of $\Pi_{r,Y}$, we get $\sct_B(\widetilde{\Pi}_{r,Y})\leq \sct_B(\Pi_{r,Y})$. Then by using Proposition~\ref{homotopic_distance_and_secat}, we obtain the following corollary.
\begin{corollary}  
    $D_B(f_1,\dots,f_r)\leq \sct_B(\Pi_{r,Y})$.
\end{corollary}
\begin{corollary} \label{homotopy_distance_to_TC}
    Let $X$ be a fibrewise space over $B$. We denote by $pr_i\colon X_B^r\to X$ the projection map corresponding to the $i$-th factor in the fibrewise product space $X_B^r$. Then \[D_B(pr_1,\dots,pr_r)= \sct_{{B}}(\Pi_{r,X}\colon P_B(X)\to X_B^r)= \TC_{B,r}(X).\]
\end{corollary}
We next derive a cohomological lower bound and a homotopy dimension--connectivity upper bound for the sequential parametrized homotopic distance. %Recall that a fibrewise space \( X \) is said to be \emph{fibrant} if its projection \( p_X\colon X \to B \) is a Hurewicz fibration.
\begin{theorem} \label{lower-upper bound for fiberwise homotopic distance}
    Let $f_1,\dots,f_r\colon X\to Y$ be fibrewise maps between two fibrant spaces $X$ and $Y$ over $B$. Then for any ring $R$, we obtain
    \begin{enumerate}
        \item Let $z_1,\dots,z_k\in H^*(Y_B^r; R)$ such that for all $1\leq i\leq k$ we have
        $$(\Delta_Y^r)^*(z_i)=0~~\text{and}~~(f_1,\dots,f_r)^*(z_1 \smile \dots \smile z_k)\neq 0,$$ 
        where $(f_1,\dots,f_r)^* \colon H^*(Y_B^r; R) \rightarrow H^*(X_B; R)$ is the induced map of the fibrewise map $(f_1,\dots,f_r) \colon X_B \rightarrow Y^r_B$.
        Then $D_B(f_1,\dots,f_r)\geq k$.
        \item If $p_Y\colon Y \to B$ is an $k$-equivalence for some $k\geq 1$, with $X$ and $Y$ path-connected and $X$ having the homotopy type of a CW complex, then  
\[
D_B(f_1,\dots,f_r) \leq \tfrac{hdim(X)}{k},
\]  
where $hdim(X)$ denotes the homotopy dimension of $X$.

    \end{enumerate}
\end{theorem}
\begin{proof}
 {(1)} From Proposition \ref{homotopic_distance_and_secat}, we have 
$D_B(f_1,\dots,f_r) = \operatorname{secat}_B(\widetilde{\Pi}_{r,Y}).$
Since $Y$ is fibrant, the map $\Pi_{r,Y} \colon P_B(Y) \to Y_B^r$ is a Hurewicz fibration, and therefore $\widetilde{\Pi}_{r,Y}$ is also a Hurewicz fibration. Moreover, because $X$ is fibrant, the space $\mathcal{P}_B(f_1,\dots,f_r)$ is fibrant as well. Hence $\widetilde{\Pi}_{r,Y}$ is a fibrewise map between fibrant spaces over $B$. By \cite[Theorem 2.10]{GC}, it follows that $\operatorname{secat}_B(\widetilde{\Pi}_{r,Y}) = \operatorname{secat}(\widetilde{\Pi}_{r,Y}).$
Since $\Pi_{r,Y}$ is a Hurewicz fibration, \cite[Corollary 1.5]{GCrelsecat} gives 
$$\sct(\widetilde{\Pi}_{r,Y}) = \sct_{(f_1,\dots,f_r)}(\Pi_{r,Y}).$$
We note that the diagonal map $\Delta_Y^r\colon Y\to Y_B^r$ fits into the following commutative diagram
\[\begin{tikzcd}
	Y && {P_B(Y)} \\
	&& {Y_B^r},
	\arrow["{i_Y}", hook, from=1-1, to=1-3]
	\arrow["{\Delta_Y^r}"', from=1-1, to=2-3]
	\arrow["{\Pi_{r,Y}}", from=1-3, to=2-3]
\end{tikzcd}\]
where $i_Y\colon Y\to P_B(Y)$ is defined as $i_Y(y)=(p_Y(y), c_y)$. This implies that $(\Delta_Y^r)^*(z_i)=0$ if and only if $(\Pi_{r,Y})^*(z_i)=0$ for all $z_i\in H^*(Y_B^r; R)$.
Finally, applying \cite[Proposition 3.1(1)]{GCrelsecat}, we conclude that  
\[
D_B(f_1,\dots,f_r) = \operatorname{secat}_{(f_1,\dots,f_r)}(\widetilde{\Pi}_{r,Y}) \geq k.
\]

{(2)} Since $p_Y \colon Y \to B$ is a $k$-equivalence, $\Pi_{r,Y}$ is a $(k-1)$-equivalence. Now using \cite[Proposition~3.1(2)]{GCrelsecat}, we get 
\[
\operatorname{secat}_{(f_1,\dots,f_r)}(\Pi_{r,Y})\leq \frac{\operatorname{hdim}(X)}{k}.
\]
Since $D_B(f_1,\dots,f_r)=\operatorname{secat}_{(f_1,\dots,f_r)}(\Pi_{r,Y})$ as shown earlier, the result follows.
\end{proof}
%As a consequence of the above theorem and Corollary \ref{homotopy_distance_to_TC}, we obtain the following.
The combination of the Theorem~\ref{lower-upper bound for fiberwise homotopic distance} and Corollary~\ref{homotopy_distance_to_TC} leads to the following remark.
\begin{remark}
 Suppose $f_i=pr_i\colon X^r_B\to X$ be projections for $1\leq i \leq r$ in Theorem~\ref{lower-upper bound for fiberwise homotopic distance}. Then the lower bound coincides with the cohomological lower bound on the sequential parametrized topological complexity given by Farber and Paul in \cite[Proposition 6.3]{Farber-Paul1}.

Moreover, suppose that $X$ is fibrant over $B$, path-connected, and has the homotopy type of a CW complex. If $p_X \colon X \to B$ is a $k$-equivalence, then $\TC_{B,r}(X) \leq \tfrac{hdim(X_B^r)}{k}$.
% Moreover, suppose that $X$ is fibrant over \( B \) such that \( X \) is path-connected and has the homotopy type of a CW complex. If \( p_X \colon X \to B \) is a \( k \)-equivalence, then $\TC_{B,r}(X) \leq \tfrac{hdim(X_B^r)}{k}$.
\end{remark}

\subsection{Relation with fibrewise unpointed LS Category}
We now describe the fibrewise unpointed LS category in terms of the sequential parametrized homotopic distance and establish a connection between them.
 %We study sequential parametrized homotopy distance through the notion of fibrewise unpointed LS category.  
\begin{remark}\label{alter_Cat_B^*(X)}
We note from the Definition \ref{fibrewise_unpointed_LS_Cat} that $\ct_B^*(X)= D_B(id_X, \underbrace{s_X\circ p_X, \dots,s_X\circ p_X}_{(r-1) \text{ times}})$.
\end{remark}
For a fibrewise pointed space $X$ over $B$, we now consider the following fibrewise map $i_j\colon X_B^{r-1}\to X^r_B$  defined by \begin{equation}\label{inclusion}
i_j(x_1,...,x_{r-1})=\left(x_1,...x_{j-1}, s_X\circ p_{X^{r-1}}(x_1,\dots,x_{r-1}),x_{j},...,x_{r-1}\right).\end{equation}
Moreover, note that $p_{X^{r-1}}(u)= p_X(pr_1(u))$ for any $u\in X_B^{r-1}$. With these notations in place, we obtain the following result. \begin{proposition}\label{interpretation of LS category interms of SPHD}
     For a fibrewise pointed space $X$ over $B$, we get $D_B(i_1,\dots,i_r)= \ct_B^*(X_B^{r-1}).$
     %\[D_B(i_1,\dots,i_r)= \ct_B^*(\overbrace{X\times_B X\times_B \dots \times_B X}^{(r-1)}).\]
\end{proposition}
\begin{proof}
Let $U\subseteq X_B^{r-1}$ be an open subset such that the following commutative diagram up to fibrewise homotopy:
\[\begin{tikzcd}
    {U} && {X_B^{r-1}} \\
    	& B.
    	\arrow["{i_{U}}", hook, from=1-1, to=1-3]
    	\arrow["{p_{X_B^{r-1}}\vert_{U}}"', from=1-1, to=2-2]
    	\arrow["{s_{X_B^{r-1}}}"', from=2-2, to=1-3]
    \end{tikzcd}
    \] 
Hence, there exists a fibrewise homotopy $H\colon  I_B(U)\to X_B^{r-1}$ such that 
    \[H(u,0)=u \text{ and } H(u,1)= s_{X_B^{r-1}}\circ p_{X_B^{r-1
}}(u) \text{ for all } u\in U.\] Note that $s_{X_B^{r-1}}\circ p_{X_B^{r-1}}(u)= \left(s_X\circ p_X(pr_1(u)),\dots,s_X\circ p_X(pr_1(u))\right)$.

Now, we define $H'\colon  I_B(U)\to X_B^r$ as follows
\[H'(u,t)=\begin{cases}
    (t_1,\dots,t_{j-1}, s_X\circ p_X(pr_1(u)), t_j,\dots, t_{r-1}) & \mbox{if } 0\leq t\leq \frac{1}{2};\\
    (w_1,\dots,w_{j},\dots, w_{k-1}, s_X\circ p_X(pr_1(u)),w_k,\dots,w_{r-1}) & \mbox{if } \frac{1}{2}\leq t\leq 1,
\end{cases}\]
where $t_l= pr_l\circ H(u,2t)$ and $w_l=pr_l\circ H(u,2-2t)$. 
Clearly, $H'(u,0)= i_j(u)$ and $H'(u,1)=i_k(u)$ for all $u\in U$. Moreover, note that
\[p_{X_B^{r-1}}\circ H'(u,t)=\begin{cases}
    p_X(t_1)= p_X\circ (pr_1\circ H(u,2t))& \mbox{if } 0\leq t\leq \frac{1}{2}\\
    p_X(w_1)= p_X\circ (pr_1\circ H(u,2-2t)) & \mbox{if } \frac{1}{2}\leq t\leq 1.
\end{cases}\]
Since $H$ is fibrewise homotopy, $p_{X_B^{r-1}}\circ H= p_X\circ pr_1\circ H\simeq_B p_{I_{B}(X)}\vert_{I_B(U)}$ and hence this implies that $p_{X_B^{r}}\circ H'(u,t)\simeq_B p_{I_B(X)}\vert_{I_B(U)}$. Therefore, $H$ is a fibrewise homotopy between $i_j\vert_U$ and $i_k\vert_{U}$ for $1\leq j\leq k\leq r$. Using the same argument applied to an open cover of $X_B^{r-1}$, we get  \[D_B(i_1,\dots,i_r)\leq \ct_B^*(X_B^{r-1}).\]

Conversely, suppose we have a fibrewise homotopy $F_j\colon  I_B(U)\to X_B^{r}$ between $i_j\vert_U$ and $i_{j+1}\vert_U$ for some open subset $U\subseteq X_B^{r-1}$ and all $1\leq j\leq r-1$.
We now define a map $\widetilde{F}\colon  I_B(U)\to X_B^{r-1}$ by
\[\widetilde{F}(u,t)=\left(pr_2(F_1(u,t)),\dots,pr_{(j+1)}(F_j(u,t)),\dots,pr_r(F_{r-1}(u,t))\right),\] where $pr_i\colon  X_B^r\to X$ denotes the projection map into the $i$-th factor. 
Clearly,
    \begin{align*}
    \widetilde{F}(u,0)
    &= \big(pr_2(F_1(u,0)),\, \dots,\, pr_{(j+1)}(F_j(u,0)),\, \dots,\, pr_r(F_{r-1}(u,0))\big) \\
    &= (x_1, \dots, x_{r-1})
\end{align*}
and
\begin{align*}
    \widetilde{F}(u,1)
    &= \big(pr_2(F_1(u,1)),\, \dots,\, pr_{(j+1)}(F_j(u,1)),\, \dots,\, pr_r(F_{r-1}(u,1))\big) \\
    &= \big(s_X \circ p_{X_B^{r-1}}(u),\, \dots,\, s_X \circ p_{X_B^{r-1}}(u)\big).
\end{align*}
This implies that $\ct_B^*(X_B^{r-1})\leq D_B(i_1,\dots,i_r)$. This completes the proof.
\end{proof}

%%%%%%%%%%%%%%%%%%%%%%%%%%%%%%%%%%%%%%%%%%%%%%%%%%%%%%%%%%%%%%%%%%%%%%%%%
\section{Properties of the sequential parametrized homotopic distance}
\label{properties of seq PHD}
This section is devoted to the study of the fundamental properties of the sequential parametrized homotopic distance, focusing in particular on its compatibility with composition of fibrewise maps and its invariance under fibrewise homotopy.
\begin{proposition}\label{Properties}\noindent
Let $f_i\colon X \rightarrow Y$ be fibrewise maps over $B$, for $i \in \{1, \dots, r\}$. Then
\begin{enumerate} 
    \item $D_B(f_i, f_j)= D_B(f_i,\overbrace{f_j,\dots,f_j}^{k})$ for all $k\geq 1$ and $i, j \in \{1, \dots, r\}$.
    \item  $D_B(f_1,\dots,f_r)= D_B(f_{\sigma(1)},\dots,f_{\sigma(r)})$ for any permutation $\sigma$ of $\{1,\dots,r\}$.
    \item $D_B(f_1,\dots,f_r)=0$ if and only if $f_i\simeq_{B} f_{i+1}$ for each $i\in \{1,\dots,r-1\}$.
    \item Given fibrewise maps $f_i\colon  X\to Y$ and $g_i\colon  X\to Y$ for $i\in \{1,\dots,n\}$. If $f_i\simeq_B g_i$ for each $i$, then $D_B(f_1,\dots,f_r)= D_B(g_1,\dots,g_r).$
    \item If $1< s < r$, then $D_B(f_1,\dots,f_s)\leq D_B(f_1,\dots,f_r).$
    \item Let $h_1\colon  Y\to Z$ and $h_2\colon  Z'\to X$ be fibrewise maps over $B$, then
    \begin{enumerate}
        \item $D_B(h_1\circ f_1,\dots,h_1\circ f_r)\leq D_B(f_1,\dots,f_r)$.
        \item $D_B(f_1\circ h_2,\dots, f_r\circ h_2)\leq D_B(f_1,\dots,f_r)$.
    \end{enumerate}
\end{enumerate}
\end{proposition}
\begin{proof}  
Statements \((1)\)--\((5)\) follow directly from Definition~\ref{seqential_parametrized_homotopic distance}.  

For \((6)(\text{a})\), let \(U \subseteq X\) be an open subset such that \(f_i|_U \simeq_B f_j|_U\) for all \(1 \leq i,j \leq r\).  
Applying \(h_1\) to each map gives \((h_1 \circ f_i)|_U \simeq_B (h_1 \circ f_j)|_U\) for all \(i,j\).  
Therefore,  
\[
D_B(h_1 \circ f_1, \dots, h_1 \circ f_r) \leq D_B(f_1, \dots, f_r).
\]  

For \((6)(\text{b})\), let \(U \subseteq X\) be an open subset such that \(f_i|_U \simeq_B f_j|_U\) for all \(1 \leq i,j \leq r\).  
Define \(V := h_2^{-1}(U) \subset Z'\). Then \(V\) is open in \(Z'\).  
Denote by \(\bar{h}_2 \colon V \to U\) the restriction of \(h_2\).  
The induced map \(\bar{h}_2 \colon I_B(V) \to I_B(U)\) allows us to transfer the homotopies, yielding  
\[
(f_i \circ h_2)|_V \simeq_B (f_j \circ h_2)|_V
\]
for all \(1 \leq i,j \leq r\).  
Hence,  $D_B(f_1 \circ h_2, \dots, f_r \circ h_2) \leq D_B(f_1, \dots, f_r).$
\end{proof}

Next, we establish the sub-additivity property for the sequential parametrized homotopic distance.

\begin{proposition}\label{sub-additivity}
 %   Let $f_1,\dots,f_n$ be fibrewise maps over $B$ and if $\{U_0,\dots,U_k\}$ be any covering of $X$, then 
 %   \[D_B(f_1,\dots,f_n)\leq \Sigma_{i=0}^k D_B(f_1\vert_{U_i},\dots,f_n\vert_{U_i})+k. \]
 Let $f_{1}, \dots, f_{r}\colon  X\to Y$ be fibrewise maps over $B$. If $\{U_{0}, \dots, U_{k}\}$ is a covering of $X$, then 
\[
D_{B}(f_{1}, \dots, f_{r}) \leq \sum_{i=0}^{k} D_{B}(f_{1}|_{U_{i}}, \dots, f_{r}|_{U_{i}}) + k.
\]
\end{proposition}

\begin{proof}
For each $i \in \{0, \dots, k\}$, let $m_i = D_B(f_{1}|_{U_i}, \dots, f_{r}|_{U_i}).$
By definition, there exists an open covering $\{U_i^0, \dots, U_i^{m_i}\}$
of $U_i$ such that 
\[
f_s|_{U_i^j} \simeq_B f_t|_{U_i^j}
\quad \text{for all } j \in \{0, \dots, m_i\}.
\] 
Consider the collection
$\mathcal{U} = \{U_0^0, \dots, U_0^{m_0}, \, U_1^0, \dots, U_1^{m_1}, \, \dots, \, U_k^0, \dots, U_k^{m_k}\},$
which covers $X$ and satisfies $f_s|_{V} \simeq_B f_t|_{V}$ for every $V \in \mathcal{U}$.  
Hence, $D_B(f_1, \dots, f_r) \leq (m_0 + \cdots + m_k) + (k+1)-1.$
This completes the proof.
\end{proof}

To state the next result, we recall the fibrewise unpointed LS category of a map. For a fibrewise pointed map $f:X\to Y$ over $B$, its fibrewise unpointed LS category, denoted by $\ct_B^*(f)$, is defined as the smallest non-negetive integer $k$ (or infinity if such $k$ does not exist) for which there exists an open cover $\{U_0,U_1,\dots,U_k\}$ of $X$ such that for each $0\leq i\leq n$, $f\vert_{U_i}\simeq_B s_Y\circ p_X\vert_{U_i}$.
In \cite[Corollary 4.3]{Calcines-Navnath}, Calcines and the first author showed that $D_B(f,g) \le (\ct_B^*(f)+1)(\ct_B^*(g)+1) - 1$ holds for fibrewise pointed maps \(f,g \colon  X \to Y\) over \(B\). The following proposition establishes a sequential analogue of this result.
%,where \(\ct_B^*(f)\) denotes the fibrewise unpointed LS category of \(f\)
\begin{proposition} \label{seq PHD-LS cat}
    Let $f_1,\dots,f_r\colon  X\to Y$ be fibrewise pointed maps over $B$. Then \[D_B(f_1,\dots,f_r)\leq \prod_{i=1}^r\left(\ct_B^*(f_i)+1\right)-1.\]
\end{proposition}
\begin{proof}
    Since $\ct_B^*(f_i)= D_B(f_i,s_Y\circ p_X)$, by Proposition~\ref{Properties} {(1)} we have 
    $$\ct_B^*(f_i)=D_B(f_i,s_Y\circ p_X,\dots,s_Y\circ p_X).$$ Suppose $\ct_B^*(f_i)=m_i$ for all $i=1,\dots,r$. Then there exists an open cover $\{U_0^i,\dots,U_{m_i}^i\}$ of $X$ such that $f\vert_{U_j^i}\simeq_B s_Y\circ p_X\vert_{U_j^i}$ for all $j=0,\dots,m_i$.
Consider the open set $W_{i_1\dots i_r}= U_{i_1}^1\cap U_{i_2}^2\cap\dots U_{i_r}^r$. Then $\{W_{i_1\dots i_r}\mid 0\leq i_j\leq m_j, 1\leq j\leq r\}$ is an open cover of $X$. Moreover, $f_s\vert_{W_{i_1\dots i_r}}\simeq_B s_Y\circ p_X\vert_{W_{i_1\dots i_r}}\simeq_B f_t\vert_{W_{i_1\dots i_r}}$ for all $1\leq s,t\leq r$. This implies that \[D_B(f_1,\dots,f_r)\leq \prod_{i=1}^r\left(m_i+1\right)-1.\] This completes the proof.
\end{proof}
\subsection{Composition of fibrewise maps}\label{composition of maps}
The following propositions describe the behaviour of the sequential parametrized homotopic distance under composition.
\begin{proposition}
    Let $f_1,\dots,f_r\colon  X\to Y$ be fibrewise maps over $B$.
    \begin{itemize}
        \item[(a)] If $h_1,\dots,h_r\colon  Y \to Z$ be fibrewise map over $B$ such that $h_i\simeq_B h_{i+1}$ for all $i=1,\dots,r-1$, then $D_B(h_1\circ f_1,\dots,h_r\circ f_r)\leq D_B(f_1,\dots,f_r).$ 
        \item[(b)] If $h_1',\dots,h_r'\colon  Z \to X$ be fibrewise map over $B$ such that $h_i'\simeq_B h_{i+1}'$ for all $i=1,\dots,r-1$, then $D_B(f_1\circ h_1',\dots,f_r\circ h_r')\leq D_B(f_1,\dots,f_r).$
    \end{itemize}
\end{proposition}
\begin{proof}
    {(a)} Let $U\subseteq X$ be an open subset with fibrewise homotopy $H_i\colon  f_i\vert_U\simeq_B f_{i+1}\vert_U$ for all $1\leq i\leq r-1$. Suppose $\widetilde{H_i}$ be the fibrewise homotopy between $h_i$ and $h_{i+1}$ for each $i=1.\dots,r-1$. Using these homotopies we define 
    \begin{equation*}\label{action_of_theta_n_on_S} 
    \begin{aligned}
    \mathcal{F}_i\colon  I_B(U) &\longrightarrow Z \\
    \hspace{3em} \left(u, t\right) &\mapsto \widetilde{H_i}\left(H_i(u,t),t\right).
    \end{aligned}
\end{equation*}
Clealy, $\mathcal{F}_i$ is a fibrewise homotopy between $(h_i\circ f_i)\vert_U$ and $(h_{i+1}\circ f_{i+1})\vert_U$. Carrying out the similar argument on an open cover of $X$ leads to
\[D_B(h_1\circ f_1,\dots,h_r\circ f_r)\leq D_B(f_1,\dots,f_r).\]

{(b)} Consider an open subset $U\subseteq X$ with fibrewise homotopy $f_i\vert_U\simeq_B f_{i+1}\vert_U$. Define the open set $V:=\bigcap_{i=1}^r (h_i')^{-1}(U)\subset Z$. Let $H_{j}\colon  V\to U$ be the restriction map of $h_j'$. Then \[(f_i\circ h_i)\vert_V = f_i\vert_U\circ H_{i}\simeq_B f_{i+1}\vert_U\circ H_{i}\simeq f_{i+1}\vert_U\circ H_{i+1}= (f_{i+1}\circ h_{i+1})\vert_V.\]

Now applying this argument to an open cover gives the required inequality. 
\end{proof}
\begin{proposition}\label{compo}
    Let $h_1,\dots,h_r\colon  Z\to X$ and $f_1,\dots,f_r\colon  X\to Y$ be fibrewise maps such that $f_i\circ h_s\simeq_B f_{i+1}\circ h_s$ for all $1\leq i\leq (r-1)$ and $2\leq s\leq r$. Then \[D_B(f_1\circ h_1,\dots,f_r\circ h_1)\leq D_B(h_1,\dots,h_r).\]
\end{proposition}
\begin{proof}
    Let $U\subseteq Z$ be an open subset such that $h_s\vert_U\simeq_B h_t\vert_U$ for all $1\leq s,t\leq r$. This implies that $f_i\circ h_i\vert_U\simeq_B f_i\circ h_{i+1}\vert_U$.
    Since $f_i\circ h_s\simeq_B f_i\circ h_s$, we have
    $f_i\circ h_i\vert_U\simeq_B f_{i+1}\circ h_{i+1}\vert_U$ for all $1\leq i\leq r-1$. Using this argument for an open cover of $Z$, we obtain the desired inequality.
\end{proof}
The preceding proposition leads to the following corollary.
\begin{corollary}\label{bound}
    Let $f_1,\dots,f_r\colon  X\to Y$ be two fibrewise pointed maps over $B$. Then \[D_B(f_1,\dots,f_r)\leq \min\{\ct_B^*(X), \TC_{B,r}(Y)\}.\]
\end{corollary}
\begin{proof}
    We note that $f_i\circ s_X\circ p_X= s_Y\circ p_X=f_{i+1}\circ s_X\circ p_X$. Now setting $h_1= Id_X$ and $h_l= s_X\circ p_X$ for all $2\leq l\leq r$ in the Proposition~\ref{compo}, we get \[D_B(f_1,\dots,f_r)\leq D_B(id_X,s_X\circ p_X,\dots,s_X\circ p_X)=\ct_B^*(X).\]

   Again from Proposition~\ref{homotopic_distance_and_secat} and the fact that \(\widetilde{\Pi}_{r,Y}\) is the fibrewise pullback of \(\Pi_{r,Y}\), it follows that 
\[
D_B(f_1,\dots,f_r) = \sct_B(\widetilde{\Pi}_{r,Y}) \le \sct_B(\Pi_{r,Y}) = \TC_{B,r}(Y).
\]
\end{proof}
\begin{remark}
Since $\TC_{B,r}(X) = D_B(pr_1, \dots, pr_r)$, the Corollary~\ref{lower-upper bound for fiberwise homotopic distance} recovers the upper bound $\TC_{B,r}(X) \le \ct_B^*(X_B^r)$ from Proposition~\ref{relation_SFTC_Cat}.
\end{remark}

We now prove that the notion of sequential parametrized homotopic distance is invariant under fibrewise homotopy.
\begin{proposition}\label{homotopic_distance_under_fibre_he}
    Let $f_1,\dots,f_r\colon  X\to Y$ be  fibrewise maps over $B$. 
    \begin{enumerate}
        \item If there exists a fibrewise map $h\colon  Y\to Y'$ with a left fibrewise homotopy inverse, then \[D_B(h\circ f_1,\dots,h\circ f_r)= D_B(f_1,\dots,f_r).\]
        \item If there exists a fibrewise map $H\colon  X'\to X$ with a right fibrewise homotopy inverse, then \[D_B(f_1\circ H,\dots,f_r\circ H)= D_B(f_1,\dots,f_r).\]
    \end{enumerate}
\end{proposition}
\begin{proof}
    Let $h'\colon  Y'\to Y$ be the left fibrewise homotopy inverse of $h$. Then $h'\circ h\simeq_B Id_Y$. Now applying Proposition \ref{Properties} {(6)}, we obtain 
    \begin{align*}
D_B(f_1,\dots,f_r) &= D_B(h'\circ h\circ f_1,\dots, h'\circ h\circ f_r) \\
&\leq D_B(h\circ f_1,\dots,h\circ f_r) \\
&\leq D_B(f_1,\dots,f_r).
\end{align*}
This completes the proof of {(1)}.
Similarly, we can prove {(2)}.
\end{proof}
The Proposition~\ref{homotopic_distance_under_fibre_he} shows that the sequential parametrized homotopic distance is invariant under fibrewise homotopy, in the following sense.
\begin{corollary}
    Let $f_1,\dots,f_r\colon  X\to Y$ and $g_1,\dots,g_r\colon  \widetilde{X}\to \widetilde{Y}$ be fibrewise maps over $B$. Assume there exist fibrewise homotopy equivalences $\alpha: \widetilde{X}\simeq_B X$ and $\eta: Y\simeq_B \widetilde{Y}$ such that $g_i\simeq_B \eta\circ f_i\circ \alpha$ for all $i=1,\dots,r$. Then $D_B(f_1,\dots,f_r)= D_B(g_1,\dots,g_r).$
\end{corollary}
\subsection{Fibrewise maps on normal spaces}\label{Fibrewise_maps_on_normal_spaces}

In \cite{Calcines-Navnath}, Garc\'ia-Calcines and the first author obtained a triangle inequality for the parametrized homotopic distance when the domain is normal. The following theorem shows that an analogous inequality holds in the sequential framework.

\begin{theorem}\label{sub-additive}
Let $X$ be a normal space and let $n,m \in \mathbb{Z}^+$ with $n \leq m$. Suppose we are given fibrewise maps 
$f_1, \dots, f_n, g_1, \dots, g_m, h_1, \dots, h_m \colon X \to Y$
over $B$. Then
\[
D_B(f_1, \dots, f_n, h_1, \dots, h_m) 
\leq D_B(f_1, \dots, f_n, g_{\sigma(1)}, \dots, g_{\sigma(s)}) 
+ D_B(g_{\beta(1)}, \dots, g_{\beta(s')}, h_1, \dots, h_m),
\]
where $\sigma$ and $\beta$ are permutations of $\{1, \dots, m\}$ and $\{1, \dots, n\}$, respectively, such that 
$g_{\sigma(i)} \simeq_B g_{\beta(j)}$ for some $i \in \{1, \dots, m\}$ and $j \in \{1, \dots, n\}$.
\end{theorem}
\begin{proof}
    Let $D_B(f_1, \dots, f_n, g_{\sigma(1)}, \dots, g_{\sigma(s)})=r_1$ and $D_B(g_{\beta(1)}, \dots, g_{\beta(s')}, h_1, \dots, h_m)=r_2$. Hence, there exists an open covering $\mathcal{U}=\{U_0,\dots,U_{r_1}\}$ of $X$ such that for each $i=0,\dots,r_1$ \[f_1\vert_{U_i}\simeq_B \dots\simeq_B f_n\vert_{U_i}\simeq_B g_{\sigma(1)}\vert_{U_i}\simeq_B\dots\simeq_B g_{\sigma(s)}\vert_{U_i}.\] 
    Similarly, there exits an open covering $\mathcal{V}=\{V_0,\dots,V_{r_2}\}$ of $X$ such that  for $j=0,\dots,r_2$ \[g_{\beta(1)}\vert_{V_j}\simeq_B \dots\simeq_B g_{\beta(s')}\vert_{V_j}\simeq_B h_1\vert_{V_j}\simeq_B \dots \simeq_B h_m\vert_{V_j}.\] 

    Since $g_{\sigma(i)}\simeq_B g_{\beta(j)}$ for some $i,j$, it follows from \cite[Lemma 4.3]{oprea2011mixing} that there exits a third open cover $\mathcal{W}=\{W_0,\dots,W_{r_1+r_2}\}$ of $X$ such that $f_1\vert_{W_l}\simeq_B\dots\simeq_B f_n\vert_{W_l}\simeq_B g_1\vert_{W_l}\simeq_B\dots\simeq_B g_m\vert_{W_l}\simeq_B h_1\vert_{W_l}\simeq_B\dots\simeq_B h_m\vert_{W_l}$ for all $l=0,\dots,r_1+r_2$. This shows that \[D_B(f_1, \dots, f_n, h_1, \dots, h_m) 
\leq r_1+r_2.\]
\end{proof}
We now state and prove an additive-type inequality for the sequential parametrized homotopic distance under composition.

\begin{proposition} \label{triangle ineq}
Let $f_1, \ldots, f_r\colon  X\to Y$ and $g_1, \ldots, g_r \colon  Y \to Z$ be fibrewise maps over $B$.  
Suppose that $X$ is a normal space. Then
\[
D_B(g_1 \circ f_1, \ldots, g_r \circ f_r) \leq D_B(f_1, \ldots, f_r) + D_B(g_1, \ldots, g_r).
\]
\end{proposition}
\begin{proof}
Suppose that $D_B(f_1,\dots,f_r)=m$ and $D_B(g_1,\dots,g_r)=n$. Then there exists an open cover $\{U_i\}_{i=0}^m$ of $X$ such that $f_1|_{U_i} \simeq_B \cdots \simeq_B f_r|_{U_i}$ holds for each  $i=0,\dots, m$,
and an open cover $\{V_j\}_{j=0}^n$ of $Y$ such that $g_1|_{V_j} \simeq_B \cdots \simeq_B g_r|_{V_j}$ holds for each  $j=0,\dots,n$.

From the above, it follows that on each $U_i$,
\[
(g_s \circ f_s)|_{U_i} \simeq_B (g_s \circ f_t)|_{U_i}, \quad 1 \leq s,t \leq r.
\]
Next, for each $j=0,\dots,n$, define 
\[
\widetilde{V_j} := \bigcap_{i=1}^r f_i^{-1}(V_j).
\]
Clearly, $\{\widetilde{V_j}\}_{j=0}^n$ is an open cover of $X$. Moreover, for every $1 \leq s,t \leq r$ and each $j$,
\[
(g_s \circ f_s)|_{\widetilde{V_j}} \simeq_B (g_t \circ f_s)|_{\widetilde{V_j}}.
\]

By applying \cite[Lemma 4.3]{oprea2011mixing}, we obtain a refinement $\{W_k\}_{k=0}^{m+n}$ of the cover of $X$ such that, for all $1 \leq s,t \leq r$ and $k=0,\dots,m+n$, thus
$(g_s \circ f_s)|_{W_k} \simeq_B (g_t \circ f_s)|_{W_k} \simeq_B (g_t \circ f_t)|_{W_k}.$
Therefore,
$D_B(g_1 \circ f_1,\dots,g_r \circ f_r) \leq m+n.$
\end{proof}

Next, we describe how the sequential parametrized homotopic distance behaves under products of fibrewise maps. Let $f_1,\ldots,f_r \colon X \to Y$ and $g \colon \widetilde{X} \to \widetilde{Y}$ be fibrewise maps over the same base space $B$. 
One can form the natural products and define the fibrewise product map 
\[
f_i \times_B g \colon X \times_B \widetilde{X} \to Y \times_B \widetilde{Y}
\] over $B$.  
If $f_s|_U \simeq_B f_t|_U$ for some open subset $U \subseteq X$ and for all $1 \leq s,t \leq r$, then  
\[
(f_s \times_B g)|_{(U \times \widetilde{X}) \cap (X \times_B \widetilde{X})} 
\simeq_B 
(f_t \times_B g)|_{(U \times \widetilde{X}) \cap (X \times_B \widetilde{X})}.
\]  
It follows that
\[
D_B(f_1 \times_B g, \ldots, f_r \times_B g) \leq D_B(f_1, \ldots, f_r).
\]
A symmetric argument gives
$D_B(g \times_B f_1, \ldots, g \times_B f_r) \leq D_B(f_1, \ldots, f_r).$
Motivated by the previous observations on fibrewise products with a fixed map, we state a corresponding inequality for fibrewise products of multiple maps.

\begin{proposition}\label{D_B_on_product}
    Let $f_1,\dots,f_r\colon  X\to Y$ and $g_1,\dots,g_r\colon  \widetilde{X}\to \widetilde{Y}$ be fibrewise maps over $B$. If $X\times_B \widetilde{X}$ is a normal space, then
    \[D_B(f_1\times_B g_1,\dots,f_r\times_B g_r)\leq D_B(f_1,\dots,f_r)+ D_B(g_1,\dots,g_r).\]
\end{proposition}

Before proving this proposition, we make the following observation. If $D_B(f_1,\dots,f_r)=n$, then there exists an open covering $\{U_i\}_{i=0}^n$ of $X$ such that on each $U_i$ there are fibrewise homotopies  
\[
H_1'\colon  f_1|_{U_i}\simeq_B f_2|_{U_i}, \;\dots,\; H_{r-1}'\colon  f_{r-1}|_{U_i}\simeq_B f_r|_{U_i}.
\]  
Fix one of these fibrewise homotopies, say $H_1'$, and define $f(x):=H_1'(x,\tfrac{1}{2})$. This gives rise to new homotopies  
\[
H_1\colon  f_1|_{U_i}\simeq_B f|_{U_i}, \;\dots,\; H_r\colon  f_r|_{U_i}\simeq_B f|_{U_i}.
\]  
Consequently, $D_B(f_1,\dots,f_r)=n$ implies the existence of an open covering $\{U_i\}_{i=0}^n$ of $X$ and a fixed map $f\colon  X\to Y$ such that  
$f_s|_{U_i}\simeq_B f|_{U_i}$, for all $s=1,\dots,r.$

\begin{proof}[Proof of Proposition~\ref{D_B_on_product}]
   Let $D_B(f_1,\dots,f_r)=n$ and $D_B(g_1,\dots,g_r)=m$. Then, by the previous discussion, there exists an open covering $\{U_i\}_{i=0}^n$ of $X$ and a fixed map $f\colon  X\to Y$ such that $f_s|_{U_i}\simeq_B f|_{U_i}$ for all $s=1,\dots,r$ and $i=0,\dots,n$. Similarly, there exists an open covering $\{V_j\}_{j=0}^m$ of $\widetilde{X}$ and a fixed map $g\colon  \widetilde{X}\to \widetilde{Y}$ such that $g_s|_{V_j}\simeq_B g|_{V_j}$ for all $s=1,\dots,r$ and $j=0,\dots,m$. For all $s=1,\dots,r$, the following properties hold:\\
   \textbf{Property A:}
On each $U_i\times_B \widetilde{X}$, there exists fibrewsie homotopies $f_s\times_B Id_{\widetilde{X}}\simeq_B f\times_B Id_{\widetilde{X}}$.\\
\textbf{Property B:}
On each $X\times_B V_j$, there exists fibrewise homotopies $Id_X\times_B g_s\simeq_B Id_X\times_B g$.

Note that the covering $\{U_i\times_B \widetilde{X}\}_{i=0}^n$ of $X\times_B \widetilde{X}$ satisfies Property A. Also $\{X\times_B V_j\}_{j=0}^{m}$ covers $X\times_B \widetilde{X}$ satisfying the Property B. Since $X\times_B \widetilde{X}$ is a normal space, it follows from \cite[Lemma 4.3]{oprea2011mixing} that there exists another open cover $\{W_i\}_{i=0}^{n+m}$ of $X\times_B \widetilde{X}$ satisfying both Property A and Property B. Hence for each $s=0,\dots,n+m$, we have fibrewise homotopies \[F_s\colon  I_B(W_s)\to Y\times_B \widetilde{Y} \text{ and } G_s\colon  I_B(W_s)\to Y\times_B \widetilde{Y}\] such that $F_s\colon  f_s\times_B Id_{\widetilde{X}}\simeq_B f\times_B Id_{\widetilde{X}} \text{ and } G_s\colon  Id_X\times_B g_s\simeq_B Id_X\times_B g$ respectively.
For each $s=0,1,\dots,r$, we define $H_s\colon  I_B(W_s)\to Y\times_B \widetilde{Y}$ by \[H_s(w,t)=\left(pr_1\circ F_s(w,t), pr_2\circ G_s(w,t)\right),\] where $pr_1\colon  Y\times_B \widetilde{Y}\to Y$ and $pr_2\colon  Y\times_B \widetilde{Y}\to \widetilde{Y}$ are the projection maps. Clearly, this is a fibrewise map over $B$ and it satisfies \[H_s(w,0)=(f_s\times_B g_s)(w) \text{ and } H_s(w,1)=(f\times_B g)(w).\] Thus, $D_B(f_1\times_B g_1,\dots, f_r\times_B g_r)\leq n+m$. This completes the proof. 
\end{proof}

\begin{remark}  
Corollary~\ref{parametrized_TC_on_product} can be recovered as a special case of Proposition~\ref{D_B_on_product}. Indeed, consider \(f_i = pr_i\colon  E_B^r \to E\) and \(g_i = pr_i\colon  \widetilde{E}_B^r \to \widetilde{E}\). Denote by \(Pr_i\colon  (E \times_B \widetilde{E})_B^r \to E \times_B \widetilde{E}\) the projection map. Then, applying Proposition~\ref{D_B_on_product} along with the commutative diagram  
  \[\begin{tikzcd}
	{(E\times_B \widetilde{E})_B^r} && {E_B^r\times_B \widetilde{E}_B^r} \\
	& {E\times_B \widetilde{E}}
	\arrow["{\cong }", from=1-1, to=1-3]
	\arrow["{Pr_i}"', from=1-1, to=2-2]
	\arrow["pr_i\times_{B} pr_i",from=1-3, to=2-2]
\end{tikzcd}\]
for \(1 \le i \le r\), one recovers the statement of the corollary.  
\end{remark}

\subsection{Fibrewise maps between fibrewise fibrations}\label{seq PHD fibrewise fibration}
This subsection is devoted to describing bounds for the sequential parametrized homotopic distance of fibre-preserving maps between fibrewise fibrations in terms of the sequential homotopic distance on individual fibres and the fibrewise unpointed LS category of the base.

\begin{theorem}\label{thm: fibrewise-fibration}
    Consider fibrewise pointed spaces $X$ and $X'$ over $B$, together with fibrewise fibrations 
    $\pi\colon E \to X$ and $\pi'\colon E' \to X'$ whose fibres are $F$ and $F'$, respectively.  
    For $i=1,\dots,r$, suppose $f_i \colon E \to E'$ are fibrewise maps, and denote by 
    $f_i^{0} = f_i|_{F} \colon F \to F'$ the induced maps on the fibres. Then
    \[
        D_B(f_1,\dots,f_r)+1
        \leq \bigl(D_B(f_1^{0},\dots,f_r^{0}) + 1\bigr) \cdot \bigl(\ct_B^*(X) + 1\bigr).
    \]
\end{theorem}

%\begin{theorem}\label{thm: fibrewise-fibration}
 %   Let $X, X'$ be fibrewise pointed spaces over $B$. Suppose $\pi\colon  E\to X$ and $\pi'\colon  E'\to X'$ be fibrewise fibrations with fibres $F$ and $F'$, respectively. Let $f_i \colon E \to E'$ be fibrewise maps, and denote by $f_i^{0} = f_i|_{F} \colon F \to F'$ the induced maps on the fibres, for $i = 1,\dots,r$. Then 
  %  \[D_B(f_1,\dots,f_r)\leq \left(D_B(f_1^0,\dots,f_r^0)+1\right). \left(\ct_B^*(X)+1\right).\]
%\end{theorem}
\begin{proof}
Suppose $\ct_B^*(X)=n$ and $D_B(f_1^0,\dots,f_r^0)=m$.
Consider $\{U_0,\dots,U_n\}$ a fibrewise categorical open cover of $X$, and let $\{V_0,\dots,V_m\}$ be an open cover of $F$ such that $f_{k}^0|_{V_j}\simeq_B f_{k+1}^0|_{V_j}$ for $0\leq j\leq m$ and $1\leq k\leq r-1$. For each $0\leq j\leq m$ and $1\leq k\leq r-1$, denote this fibrewise homotopy by $F_j^k\colon I_B(V_j)\to F'$.

For a fibrewise categorical open subset $U$ of $X$, there exists a fibrewise homotopy $H \colon I_B(U)\to X$ with $H\colon i_{U}\simeq_B s_X\circ p_{X}|_{U}$. Define $U':=\pi^{-1}(U)$. Then, applying the fibrewise homotopy lifting property for $\pi$, we obtain a lifted homotopy $\tilde{H}\colon I_B(U')\to E$ that makes the following diagram commute in the category of fibrewise spaces:
$$\xymatrix{
{U'} \ar[d]_{i_0} \ar@{^{(}->}@<-2pt>[rrr] & & &  {E} \ar[d]^{\pi } \\
{I_B(U')} \ar@{.>}[urrr]^{\tilde{H}} \ar[rr]_{I_B(\pi )} & & {I_B(U)} \ar[r]_(.6){H} & {X}.
}$$
Since $\pi\circ \tilde{H} =H\circ I_B(\pi )$, we have $\pi (\tilde{H}(x,1))=H(\pi(x),1)=(s_X\circ p_X)(\pi(x))\in s_X(B).$ In other words,
    $\tilde{H}(x,1)\in \pi^{-1}(s_X(B))=F$.
This defines a fibrewise map $\tilde{H}_{1}:=\tilde{H}(-,1)\colon U'\to F$.
We denote this map by $\tilde{H}_{i, 1}$ associated to such an open subset $U_i$.

Now, for $i\in \{0,1,\dots, n\}$ and $j\in \{0,1,\dots, m\}$, we consider the open set $W_{i,j}=U_i'\cap V_j'$ in $E$, where $U'_i= \pi^{-1}(U_i)$ and $V_j'=\tilde{H}_{i,1}^{-1}(V_j)$. Clearly, the collection $\{W_{i,j}\mid 0\leq i\leq n,\ 0\leq j\leq m\}$ forms an open cover of $E$.

It remains to verify that \( f_k|_{W_{i,j}} \simeq_B f_{k+1}|_{W_{i,j}} \) for all \( i\in\{0,1,\dots,n\} \), \( j\in\{0,1,\dots,m\} \), and \( k\in\{1,\dots,r-1\} \). For this, we define a map
\[
G_{i,j}^k \colon I_B(W_{i,j}) \to E'
\]
by
\[
G_{i,j}^k(e,t)=
\begin{cases}
f_k\!\left(\tilde{H}(e,3t)\right), & 0\le t\le \frac13,\\[4pt]
F_j^k\!\left(\tilde{H}_{i,1}(e),\,3t-1\right), & \frac13\le t\le \frac23,\\[4pt]
f_{k+1}\!\left(\tilde{H}_i(e,3-3t)\right), & \frac23\le t\le 1.
\end{cases}
\]
This defines a fibrewise homotopy over \(B\) from \(f_k|_{W_{i,j}}\) to \(f_{k+1}|_{W_{i,j}}\), completing the proof.
\end{proof}

If $\pi\colon  E\to X$ be a fibrewise fibration between fibrewise pointed spaces $E$ and $X$ over $B$, then Garc\'ia-Calcines and the first author \cite[Corollary 6.3]{Calcines-Navnath} proved the inequality
\[\TC_B(E)+1\leq \left(\TC_B(F)+1\right). (\ct_B^*(X\times_B X)+1),\] where $F$ is the fibre of $\pi$.

The following corollary provides its sequential analogue.

\begin{corollary}\label{cor:cib-fibration2}
Let $E$ and $X$ be fibrewise pointed spaces over $B$ and let $\pi\colon E\to X$ be a fibrewise fibration with fibre $F$. Then
$$\TC_{B,r}(E)+1\leq (\TC_{B,r}(F)+1)\cdot (\ct_B^*(X_B^r)+1).$$
\end{corollary}

\begin{proof}
We note that the projection maps $pr_i\colon  X_B^r\to X$ satisfy the following commutative diagram:
\[\begin{tikzcd}
	{E_B^r} &&& E \\
	{X_B^r} &&& {X.}
	\arrow["{pr_1,\dots,pr_r}", from=1-1, to=1-4]
	\arrow["{\pi^r_B}"', from=1-1, to=2-1]
	\arrow["\pi", from=1-4, to=2-4]
	\arrow["{pr_1,\dots,pr_r}"', from=2-1, to=2-4]
\end{tikzcd}\]
Observe that $f_i^0=pr_i\colon F_B^r\to F$. Now, the desired inequality follows from Theorem \ref{thm: fibrewise-fibration}.
\end{proof}

\section{Sequential  pointed parametrized homotopic distance and its comparison with unpointed version} \label{seq Pointed PHD}
In this section, we introduce the pointed version of the sequential parametrized homotopic distance and compare it with its unpointed counterpart. We begin with the definition of the sequential parametrized pointed homotopic distance.

\subsection{Sequential  pointed parametrized homotopic distance}
\begin{definition}\label{SPPHD}
    Let $f_1,\dots,f_r\colon  X \to Y$ be fibrewise pointed maps between fibrewise pointed spaces $X$ and $Y$ over $B$. The sequential parametrized pointed homotopic distance $D_B^B(f_1,\dots,f_r)$ is the least integer $n \geq 0$ for which there exists an open cover $\{U_0,\dots,U_n\}$ of $X$ such that 
    \[
    s_X(B) \subset U_i \quad \text{and} \quad f_1\vert_{U_i} \simeq_B^B \cdots \simeq_B^B f_r\vert_{U_i},\text{ for all } 0 \leq i \leq n.
    \]
    If no such cover exists, we set $D_B^B(f_1,\dots,f_r) = \infty$.
\end{definition}
It follows from the Definition~\ref{SPPHD} that
\begin{enumerate}
    \item $D_B^B(f_1,\dots,f_r)=D_B^B(f_{\sigma(1)},\dots,f_{\sigma(r)})$ for any permutation $\sigma$ of $\{1,\dots,r\}$.
    \item $D_B^B(f_1,\dots,f_r)=0$ if and only if $f_i\simeq_B^B f_{i+1}$ for each $i\in \{1,\dots,r-1\}$.
    \item If $f_i\simeq_B^B f_i'$ for each $i=1,2,\dots,r$, then $D_B^B(f_1,\dots,f_r)= D_B^B(f_1',\dots,f_r')$.
\end{enumerate}
As in the unpointed case, the sequential parametrized pointed homotopic distance can be related to the fibrewise pointed sectional category. Before establishing this relation, we note that if $f_1,\dots,f_r \colon X \to Y$ are fibrewise pointed maps over $B$, then, as in the diagram \eqref{Pullback_Diagram}, one obtains a map
\[
\widetilde{\Pi}_{r,Y}^B \colon \mathcal{P}_B^B(f_1,\dots,f_r) \to X,
\]
defined as the pullback of $\Pi_{r,Y}^B \colon P_B(Y) \to Y_B^{\,r}$ along $(f_1,\dots,f_r)$. Since $\Pi_{r,Y}^B$ is a fibrewise pointed fibration, the map $\widetilde{\Pi}_{r,Y}^B$ is also a fibrewise pointed fibration.

In analogy with the unpointed version of sequential parametrized homotopic distance, the pointed version can also be expressed as the fibrewise pointed sectional category of $\widetilde{\Pi}_{r,Y}^B$, as stated below.

\begin{proposition}\label{comparision}
    Let $f_1,\dots,f_r\colon  X\to Y$ be fibrewise pointed maps between fibrewise pointed spaces $X$ and $Y$ over $B$. Then \[D_B^B(f_1,\dots,f_r)= \sct_B^B(\widetilde{\Pi}_{r,Y}^B),\] where $\sct_B^B(\widetilde{\Pi}_{r,Y}^B)$ denotes the fibrewise pointed sectional category of a fibrewise pointed map $\widetilde{\Pi}_{r,Y}^B$ introduced in \cite[Definition 3.5]{GC}.
\end{proposition}
\begin{proof}
    The proof is identical to that of Proposition \ref{homotopic_distance_and_secat}.
\end{proof}
As a direct consequence of the Proposition~\ref{comparision}, we obtain the following.
\begin{corollary}
    Let $X$ be a fibrewise pointed space over $B$. Then \[D_B^B(pr_1,\dots,pr_r)= \TC_{B,r}^B(X).\] %\sct_B^B(\Pi_{r,X}\colon  P_B(X)\to X_B^r)
\end{corollary}
Analogous to the unpointed case, one can define a numerical invariant, the fibrewise pointed LS category, which can be expressed in terms of the sequential parametrized homotopic distance. 

\begin{proposition}
    Let $X$ be a fibrewise pointed space over $B$. Then 
    \[
        D_B^B(i_1, \dots, i_r) = \ct_B^B(X_B^{r-1}).
    \]
\end{proposition}

All the results presented in Section \ref{properties of seq PHD} for the sequential parametrized homotopic distance remain valid in the pointed setting as well. The proofs proceed analogously to the unpointed case, requiring only minor modifications.

We now state a cohomological lower bound for the pointed sequential parametrized homotopic distance. Combining \cite[Theorem 7.7]{Calcines-Navnath} with Proposition \ref{comparision}, we obtain the following result.

\begin{proposition}\label{seq_para_pointe_homotopic_distance_LB}
    Let $f_1,\dots,f_r\colon  X\to Y$ be fibrewise pointed maps between fibrewise pointed spaces. Then \[nil(\mathrm{ker}((\widetilde{\Pi}_{r,Y}^B)^*))\leq D_B^B(f_1,\dots,f_r).\]
\end{proposition}
 As a consequence of the preceding proposition, we obtain the following cohomological lower bound for the sequential fibrewise pointed topological complexity.
\begin{corollary}
    Let $X$ be a fibrewise pointed space over $B$ and let $\Delta_X^r\colon  X\to X_B^r$ be the diagonal map. Then \[nil\left(\mathrm{ker}((\Delta_X^r)^*\colon  H_B^*(X_B^r)\to H_B^*(X))\right)\leq \sct_B^B\left(\Pi_{r,X}^B\colon  P_B^B(X)\to X_B^r\right)= \TC_{B,r}^B(X).\]
\end{corollary}
\begin{proof}
Since $\sct_B^B(\Pi_{r,X}^B\colon  P_B^B(X)\to X_B^r) = D_B^B(pr_1, \dots, pr_r)$, the pullback \(\widetilde{\Pi}_{r,X}^B\) coincides with \(\Pi_{r,X}^B\). Moreover, there is a commutative diagram up to pointed fibrewise homotopy:

 \[\begin{tikzcd}
	X && {P_B^B(X)} \\
	& {X_B^r.}
	\arrow["{\gamma_X}","\simeq_B^B"', from=1-1, to=1-3]
	\arrow["{\Delta_X^r}"', from=1-1, to=2-2]
	\arrow["{\Pi_{r,X}^B}", from=1-3, to=2-2]
\end{tikzcd}\]
Here, \(\gamma_X\) is the fibrewise pointed homotopy equivalence defined by mapping \(x \in X\) to \((p_X(x), c_x)\), where \(c_x\) denotes the constant path at \(x\).
This implies that $\mathrm{ker}((\Delta_X^r)^*) = \mathrm{ker}((\Pi_{r,X}^B)^*)$.
The result then follows directly from Proposition \ref{seq_para_pointe_homotopic_distance_LB}.
\end{proof}

\subsection{The comparison}\label{seq PHD_comparision}
This section is devoted to the comparison of the two notions of sequential parametrized homotopic distance. We first establish general inequalities relating the pointed and unpointed versions, showing that the pointed invariant can exceed the unpointed one by at most one. Under suitable hypotheses on the base and the fibrations involved, we then show that the two invariants coincide.

\begin{theorem}
Let \(f_1, \dots, f_r \colon X \to Y\) be fibrewise pointed maps between fibrant spaces over $B$. 
If, in addition, \(X\) and \(Y\) are ANR spaces and \(B\) is locally equiconnected, then
\[
D_B(f_1, \dots, f_r) \leq D_B^B(f_1, \dots, f_r) \leq D_B(f_1, \dots, f_r) + 1.
\]
\end{theorem}

\begin{proof}
The inequality $D_B(f_1, \dots, f_r) \leq D_B^B(f_1, \dots, f_r)$ is immediate.

Since the fibrewise pointed spaces \(X\) and \(Y\) are both fibrant ANR spaces, it follows from \cite[Lemmas 8.1 and 8.4]{Calcines-Navnath} that \(X\) and \(P_B^B(Y)\) are fibrewise well-pointed. Moreover, by \cite[Proposition 3.2]{GC}, the space \(Y_B^r\) is also fibrewise well-pointed. Hence, using the same proposition, we conclude that \(\mathcal{P}_B^B(f_1, \dots, f_r)\) is fibrewise well-pointed as well. 

Finally, combining Propositions~\ref{homotopic_distance_and_secat} and~\ref{comparision}, the desired inequality follows from \cite[Theorem 4.1]{GC}.
\end{proof}

We next identify conditions under which the pointed and unpointed versions of the sequential parametrized homotopic distance coincide.

\begin{theorem}\label{Comparision_hd}
    Let $f_1,\dots,f_r\colon X\to Y$ be fibrewise pointed maps between fibrant and cofibrant fibrewise pointed spaces. Additionally, suppose that $B$ is a CW-complex, the map $p_Y\colon  Y\to B$ is a $k$-equivalence for some integer $k\geq 1$ and
        $\text{dim}(B)< (D_B(f_1,\dots,f_r)+1). k-1$. Then \[D_B(f_1,\dots,f_r)=D_B^B(f_1,\dots,f_r).\]
\end{theorem}
\begin{proof}  
Since \(Y\) is fibrant, the map \(\Pi_{r,Y} \colon P_B(Y) \to Y_B^r\) is a Hurewicz fibration. Consequently, \(\widetilde{\Pi}_{r,Y}\) is also a Hurewicz fibration. Moreover, as \(X\) is fibrant, the space \(\mathcal{P}_B(f_1, \dots, f_r)\) is fibrant as well.  
From \cite[Lemmas 8.7 and 8.8]{Calcines-Navnath}, both \(P_B^B(Y)\) and \(Y_B^r\) are cofibrant. In addition, \cite[Lemma 8.8]{Calcines-Navnath} ensures that \(\mathcal{P}_B^B(f_1, \dots, f_r)\) is also cofibrant.

Since $p_Y \colon Y \to B$ is a $k$-equivalence, the maps $\Pi_{r,Y}$ and $\Pi_{r,Y}^B$ are $(k-1)$-equivalences. Consequently, $\widetilde{\Pi}_{r,Y}$ and $\widetilde{\Pi}_{r,Y}^B$ are also $(k-1)$-equivalences. We also note that $\sct_B(\widetilde{\Pi}_{r,Y}) = \sct_B(\widetilde{\Pi}_{r,Y}^B)$. Together with the preceding observations and the condition
\[
\dim(B) < (D_B(f_1,\dots,f_r)+1)k-1 \le (\sct_B(\widetilde{\Pi}_{r,Y}^B)+1)k-1
\]
with \cite[Theorem 4.4]{GC} then yields
\[
\sct_B(\widetilde{\Pi}_{r,Y}^B)=\sct_B^B(\widetilde{\Pi}_{r,Y}^B).
\]

Finally, the result follows from Propositions \ref{homotopic_distance_and_secat} and \ref{comparision}.  
\end{proof}
As a consequence of the Theorem~\ref{Comparision_hd}, we obtain the following corollary.
\begin{corollary}
    Let $E$ be a fibrewise pointed space over $B$ which is fibrewise locally equiconnected and fibrant. Additionally, assume that $B$ is a CW-complex and the following conditions hold:
    \begin{enumerate}
        \item[(a)] $\Delta_E^r\colon  E\to E_B^r$ is a $k$-equivalence for some $k\geq 1$;
        \item[(b)] $\text{dim}(B)< (\TC_{B,r}(E)+1). k-1$.
    \end{enumerate}
    Then $\TC_{B,r}(E)=\TC_{B,r}^B(E)$.
\end{corollary}

\section{Sequential parametrized subspace homotopic distance} \label{Relative seq PHD}
In this section, we introduce the notion of subspace distance, that is, the homotopic distance between two maps relative to a subspace, as a generalization of the sequential parametrized homotopic distance given in Definition \ref{seqential_parametrized_homotopic distance}, and discuss some of its basic properties.

\begin{definition}
Let $f_1, \dots, f_r \colon  X \to Y$ be fibrewise maps between fibrewise spaces over $B$, and let $A \subseteq X$ be a subspace. 
The sequential parametrized subspace homotopic distance of $f_1, \dots, f_r$ on $A$, denoted by $D_{B,X}(A; f_1, \dots, f_r)$, is defined as the sequential parametrized homotopic distance between the restrictions $f_1|_A, \dots, f_r|_A$; that is,
\[
D_{B,X}(A; f_1, \dots, f_r) := D_B(f_1|_A, \dots, f_r|_A).
\]
\end{definition}
Note that when $A=X$, we recover the usual sequential analogue of parametrized distance.\\
The next example may be regarded as a subspace version of Corollary~\ref{homotopy_distance_to_TC}.

\begin{example}
Let $A \subseteq X_B^r$, and let $pr_i \colon  X_B^r \to X$ denote the projection onto the $i$th factor. 
Then the sequential parametrized subspace topological complexity of $A$ satisfies $\TC_{r,E}^B(A) = D_{B,X_B^r}(A; pr_1, \dots, pr_r)$.
\end{example}
As a consequence of the definition of the sequential parametrized subspace homotopic distance, we obtain the following.
\begin{proposition}
    Let $f_1,\dots,f_r\colon  X\to Y$ be fibrewise maps over $B$. Suppose $A\subseteq X$ be a subspace. Then
    \begin{enumerate}
        \item $D_{B,X}(A;f_1,\dots,f_r)\leq D_B(f_1,\dots,f_r)$.
        \item If $A\subseteq C\subseteq X$, then $D_{B,X}(A;f_1,\dots,f_r)\leq D_{B,X}(C;f_1,\dots,f_r)$.
        \item $D_{B,X}(A;f_1,\dots,f_r)= D_{B,X}(A; f_{\sigma(1)},\dots, f_{\sigma(r)})$ for any permutation $\sigma$ of $\{1,\dots,r\}$.
        \item $D_{B,X}(A;f_1,\dots,f_r)=0$ if and only if $f_i\vert_A\simeq_B f_{i+1}\vert_A$ for each $1\leq i\leq r-1$.
        \item Suppose $g_1,\dots,g_r\colon  X\to Y$ be fibrewise maps over $B$ such that $f_i\vert_A\simeq_B g_i\vert_A$. Then $D_{B,X}(A;f_1,\dots,f_r)= D_{B,X}(A;g_1,\dots,g_r)$.
        \item $D_{B,X}(A;f_1,\dots,f_s)\leq D_{B,X}(A;f_1,\dots,f_r)$ for all $1\leq s\leq r$.
    \end{enumerate}
\end{proposition}

We next note that the behaviour of the sequential parametrized homotopic distance can be clarified by examining its value on the connected components of the domain.
\begin{proposition}
    Let $f_1,\dots,f_r\colon  X\to Y$ be fibrewise maps over $B$ and let $\{A_i\}_{i=1}^n$ be the connected components of $X$. Then $D_B(f_1,\dots,f_r)=\text{max}_i\{ D_{B,X}(A_i;f_1,\dots,f_r)\}$.
\end{proposition}
\begin{proof}
    The proof follows in the same way as the proof of \cite[Proposition 3.9]{Relative_homotopic_distance}.
\end{proof}
Now we show that, as in the case of the sequential parametrized homotopic distance, the subspace version is likewise fibrewise homotopy invariant.
\begin{proposition}
Let $f_1, \dots, f_r \colon  X \to Y$ be fibrewise maps over $B$. Suppose $i_A \colon  A \hookrightarrow X$ and $i_C \colon  C \hookrightarrow X$ are inclusions of subspaces, and let $\alpha \colon  A \to C$ be a fibrewise homotopy equivalence such that $i_C \circ \alpha \simeq_B i_A$. Then 
\[
D_{B,X}(A; f_1, \dots, f_r) = D_{B,X}(C; f_1, \dots, f_r).
\]
\end{proposition}
\begin{proof}
By definition, 
$D_{B,X}(C; f_1, \dots, f_r) 
    = D_B(f_1|_C, \dots, f_r|_C)
    = D_B(f_1 \circ i_C, \dots, f_r \circ i_C).$
Since $\alpha$ is a fibrewise homotopy equivalence, Proposition~\ref{homotopic_distance_under_fibre_he} gives
\[
D_B(f_1 \circ i_C, \dots, f_r \circ i_C)
    = D_B(f_1 \circ i_C \circ \alpha, \dots, f_r \circ i_C \circ \alpha).
\]
Using the hypothesis $i_C \circ \alpha \simeq_B i_A$, we obtain from Proposition \ref{Properties} that
\[
D_B(f_1 \circ i_C \circ \alpha, \dots, f_r \circ i_C \circ \alpha)
    = D_B(f_1 \circ i_A, \dots, f_r \circ i_A)
    = D_B(f_1|_A, \dots, f_r|_A).
\]
Hence,
\[
D_{B,X}(C; f_1, \dots, f_r)
    = D_{B,X}(A; f_1, \dots, f_r).
\]
\end{proof}
We now derive the subadditive property for the sequential parametrized subspace homotopic distance.
\begin{proposition}
    Let $f_1, \dots, f_r \colon  X \to Y$ be fibrewise maps over $B$. Suppose $\{V_0,V_1,\dots,V_k\}$ be an open covering of $X$. Then \[D_B(f_1,\dots,f_r)+1\leq \sum_{i=0}^k \left(D_{B,X}(V_i;f_1,\dots,f_r)+1\right).\] 
\end{proposition}
\begin{proof}
   The proof proceeds along the same lines as that of Proposition~\ref{sub-additivity}.
\end{proof}

We now establish that the sequential parametrized subspace homotopic distance also admits a description in terms of the fibrewise sectional category.

Suppose $f_1,\dots,f_r\colon  X\to Y$ be fibrewise maps over $B$, we consider the restriction of the map $\widetilde{\Pi}_{r,Y}\colon \mathcal{P}_B(f_1,\dots,f_r)\to X$ to the preimage of a subspace $A\subseteq X$. Namely, we define 
\[\widetilde{\Pi}_{r,A}:=\widetilde{\Pi}_{r,Y}\vert_{(\widetilde{\Pi}_{r,Y})^{-1}(A)}\colon  (\widetilde{\Pi}_{r,Y})^{-1}(A)\to A\]

Since $\Pi_{r,Y}$ in \eqref{Pullback_Diagram} is a fibrewise fibration, so is $\widetilde{\Pi}_{r,A}$. With this setup, we obtain the following result.
\begin{proposition}\label{Subspace_SPHD_vs_aecat}
    Suppose $f_1,\dots,f_r\colon X\to Y$ be fibrewise maps over $B$ and $A\subseteq X$ be a subspace. Then \[D_{B,X}(A;f_1,\dots,f_r)=\sct_B(\widetilde{\Pi}_{r,A}).\]
\end{proposition}
\begin{proof}
    The proof follows the same method as in Proposition \ref{homotopic_distance_and_secat}.
\end{proof}
Analogously to Theorem \ref{lower-upper bound for fiberwise homotopic distance}, we obtain corresponding cohomological lower bounds and homotopy dimension -connectivity upper bounds for the sequential parametrized subspace homotopic distance.
\begin{theorem}
    Let $f_1,\dots,f_r\colon  X\to Y$ be fibrewise maps over $B$ and $A\subseteq X$ be a subspace. Moreover suppose that $A$ and $Y$ are fibrant spaces. Then
    \begin{itemize}
        \item[(a)] If for $z_1,\dots,z_k\in H^*(Y_B^r;R)$, $(\Delta_Y^r)^*(z_i)=0$ and $(f_1\vert_A,\dots,f_r\vert_A)^*(z_1 \smile \dots \smile z_k)\neq 0$, then $D_{B,X}(A;f_1,\dots,f_r)\geq k$. 
        \item[(b)] Suppose $p_Y\colon  Y\to B$ is an $k$-equivalence ($k\geq 1$) such that both $A$ and $Y$ are path-connected and $A$ has homotopy type of CW complex. Then $D_{B,X}(A;f_1,\dots,f_r)\leq \frac{hdim(A)}{k}$.
    \end{itemize}
\end{theorem}
\begin{proof}
    {(a)} By combining the Proposition \ref{Subspace_SPHD_vs_aecat} with \cite[Theorem 2.10]{GC} and \cite[Corollary 1.5]{GCrelsecat}, we obtain \begin{equation}\label{subspace_SPHD_identity}
        D_{B,X}(A;f_1,\dots,f_r)=\sct_{(f_1\vert_A,\dots,f_r\vert_A)}(\Pi_{r,Y}).
    \end{equation} From this identity, the argument used in Theorem~\ref{lower-upper bound for fiberwise homotopic distance} applies verbatim and yields the desired conclusion.

    {(b)} The conclusion follows directly from \eqref{subspace_SPHD_identity}, together with \cite[Proposition 3.1 (2)]{GCrelsecat}. 
\end{proof}
In parallel with Proposition~\ref{interpretation of LS category interms of SPHD}, the fibrewise unpointed subspace LS category admits an analogous description via the sequential parametrized subspace homotopic distance. Since the argument is essentially the same as in Proposition~\ref{interpretation of LS category interms of SPHD}, we do not repeat it here.

\begin{proposition}
Let \(X\) be a fibrewise space over \(B\), and let \(A \subseteq X_{B}^{\,r-1}\) be a subspace. Then  
\[
D_{B, X_{B}^{\,r-1}}(A; i_1,\dots,i_r)
= \ct_{B, E_{B}^{\,r-1}}(A),
\]
where \(i_j \colon  X_{B}^{\,r-1} \to X_{B}^{\,r}\) is the map defined in \eqref{inclusion}.
\end{proposition}

Next, we show the fibrewise homotopy invariance property of the sequential parametrized subspace homotopic distance.
\begin{proposition}
    Let $f_1,\dots,f_r\colon  X\to Y$ be fibrewise maps over $B$ and $A\subseteq X$ be a subspace.
    \begin{itemize}
        \item[(i)] Suppose the map $h\colon  Y\to Y'$ has a left fibrewise homotopy inverse. Then $D_{B,X}(A;h\circ f_1,\dots,h\circ f_r)= D_{B,X}(A;f_1,\dots,f_r)$.
        \item[(ii)] Suppose for a fibrewise map $H\colon  X'\to X$, there exists a map $H'\colon  X\to X'$ such that $H\circ H'\vert_A\simeq_B id_A$ . Then $D_{B,X}(H'(A);f_1\circ H,\dots, f_r\circ H)= D_{B,X}(A;f_1,\dots,f_r)$.
    \end{itemize}
\end{proposition}

We note that the subspace analogue of the Theorem~\ref{sub-additive} also holds for the invariant \(D_{B,X}(A;\,\cdot)\). Since the argument is obtained by restricting all maps and homotopies to the subspace \(A\subseteq X\), the proof is entirely parallel to the absolute case, and we therefore omit the statement.

In the next proposition, we show that a product inequality also holds for the subspace version of the sequential parametrized homotopic distance.
\begin{proposition}\label{D_{B,X}_on_product}
    Let $f_1,\dots,f_r\colon  X\to Y$ and $g_1,\dots,g_r\colon  \widetilde{X}\to \widetilde{Y}$ be fibrewise maps over $B$ and $A, \widetilde{A}$ are subspaces of $X$ and $\widetilde{X}$, respectively. If $A\times_B \widetilde{A}$ is a normal space, then
    \[D_{B,X\times_B \widetilde{X}}(A\times_B \widetilde{A};f_1\times_B g_1,\dots,f_r\times_B g_r)\leq D_{B,X}(A;f_1,\dots,f_r)+ D_{B,\widetilde{X}}(\widetilde{A};g_1,\dots,g_r).\]
\end{proposition}
\begin{proof}
The proof follows the same argument as Proposition~\ref{D_B_on_product}.
\end{proof}

\vspace{1cm}

\noindent \textbf{Acknowledgment}:
N. Daundkar gratefully acknowledge the support of DST–INSPIRE Faculty Fellowship (Faculty Registration No. IFA24-MA218), as well as Industrial Consultancy and Sponsored Research (IC\&SR), Indian Institute of Technology Madras for the New Faculty Initiation Grant (RF25261395MANFIG009294). Ankur Sarkar was supported by the Centre for Operator Algebras, Geometry, Matter and Spacetime, Ministry of Education, Government of India through Indian Institute of Technology Madras [Project no. SB22231267MAETWO008573].

\bibliographystyle{plain} 
\bibliography{references}

\end{document}